\documentclass[11pt]{amsart}
\pdfoutput=1
\usepackage{t1enc}
\usepackage[latin1]{inputenc}
\usepackage{graphicx}
\usepackage{emptypage}
\usepackage{amsthm}
\usepackage{amsmath}
\usepackage{amssymb}
\usepackage{sseq}
\usepackage[pagebackref]{hyperref}
\usepackage[curve]{xypic}
\usepackage[left=3cm,right=3cm,top=3cm,bottom=3cm]{geometry} 
\usepackage{pdfpages}
\usepackage{stmaryrd}
\usepackage{marvosym}
\usepackage{tikz}
  \usetikzlibrary{matrix,arrows}

\newtheorem{thm}{Theorem}[section]
\newtheorem*{thmA}{Theorem A}
\newtheorem*{thmB}{Theorem B}
\newtheorem*{thmC}{Theorem C}

\newtheorem{lemma}[thm]{Lemma}

\newtheorem{cor}[thm]{Corollary}

\newtheorem{prop}[thm]{Proposition}

\newtheorem{question}[thm]{Question}

\theoremstyle{definition}
\newtheorem{defi}[thm]{Definition}

 \newtheorem{example}[thm]{Example}

\theoremstyle{remark}

\newtheorem{remark}[thm]{Remark}

\newcommand{\G}{\mathbb{G}}
\newcommand{\Z}{\mathbb{Z}}
\newcommand{\F}{\mathbb{F}}

\newcommand{\A}{\mathbb{A}}
\newcommand{\N}{\mathbb{N}}

\newcommand{\m}{\mathfrak{m}}

\newcommand{\Ac}{\mathcal{A}}

\newcommand{\EE}{\mathcal{E}}
\newcommand{\FF}{\mathcal{F}}
\newcommand{\GG}{\mathcal{G}}
\newcommand{\II}{\mathcal{I}}

\newcommand{\LL}{\mathcal{L}}
\newcommand{\PP}{\mathcal{P}}

\newcommand{\XX}{\mathcal{X}}

\newcommand{\OO}{\mathcal{O}}

\newcommand{\MM}{\mathcal{M}}

\newcommand{\MMb}{\overline{\mathcal{M}}}

\newcommand{\at}{\widetilde{\alpha}}
\newcommand{\tensor}{\otimes}

\DeclareMathOperator{\Vect}{Vect}

\DeclareMathOperator{\fMod}{fMod}

\DeclareMathOperator{\Spf}{Spf}

\DeclareMathOperator{\grcomod}{-grcomod}
\DeclareMathOperator{\modules}{\text{-}mod}

\DeclareMathOperator{\pr}{pr}

\DeclareMathOperator{\Ext}{Ext}
\DeclareMathOperator{\Hom}{Hom}
\DeclareMathOperator{\Spec}{Spec}

\DeclareMathOperator{\QCoh}{QCoh}
\DeclareMathOperator{\Coh}{Coh}

\DeclareMathOperator{\Pic}{Pic}

\DeclareMathOperator{\Mod}{Mod}
\DeclareMathOperator{\grmodules}{\text{-}grmod}
\DeclareMathOperator{\grmod}{\text{-}grmod}
\DeclareMathOperator{\rk}{rk}
\DeclareMathOperator{\res}{res}
\DeclareMathOperator{\ind}{ind}

\DeclareMathOperator{\cart}{Cart}

\def\co{\colon\thinspace}

\mathchardef\mhyphen="2D

\begin{document}
\title{Vector Bundles on the Moduli Stack of Elliptic Curves}
\author{Lennart Meier}
\maketitle

\begin{abstract}We study vector bundles on the moduli stack of elliptic curves over a local ring $R$. If $R$ is a field or a discrete valuation ring of (residue) characteristic not $2$ or $3$, all these vector bundles are sums of line bundles. For $R = \Z_{(3)}$, we construct higher rank indecomposable vector bundles and give a classification of vector bundles that are iterated extensions of line bundles. If $R = \Z_{(2)}$, we show that there are even indecomposable vector bundles of arbitrary high rank. 
\end{abstract}

\section{Introduction and Statement of Results}

Denote by $\MM_R$ the (uncompactified) moduli stack of elliptic curves over some (commutative) ring $R$. The aim of this note is to study vector bundles on $\MM_R$. This extends the classification of line bundles by Fulton and Olsson, who prove:
\begin{thm}[\cite{F-O10}]\label{ClassLine}If $R$ is a reduced ring or a $\Z[\frac12]$-algebra, then 
\begin{align*} \Z/12\times \Pic(\A^1_R) &\to \Pic(\MM_R) \\
([k], \EE) &\mapsto \omega^k \tensor \pi^*\EE \end{align*}
is an isomorphism for $\pi\co \MM_R \to \A^1_R$ the map defined by the $j$-invariant (which is isomorphic to the canonical map into the coarse moduli space).\end{thm}
Here, $\omega$ denotes the line bundle on $\MM_R$ defined by the following property: Let $p\co E\to X$ be an elliptic curve over a scheme $X$. Then the evaluation $\omega_X$ at the map $X \to \MM_R$ classifying $E$ is given by $p_*\Omega^1_{E/X}$. 

Note that we have the following special case of Theorem \ref{ClassLine}:
\begin{cor}If $R$ is a regular local ring, the map $\Z/12 \to \Pic(\MM_R)$, $[k] \mapsto \omega^k$, is an isomorphism.\end{cor}                                                                                                                              
\begin{proof}By \cite{Har77}, II.6.6 and II.6.15, we have $\Pic(\A^1_R) \cong \Pic(\Spec R)$.\end{proof}

We will take some first steps towards the study of vector bundles of higher rank on the moduli stack of elliptic curves. We will be able to give a complete classification only under severe restriction on the local ring $R$, but have several partial results. Before we go into the details, we want to fix the definition of a vector bundle we want to use:
\begin{defi}Let $(\XX,\OO_\XX)$ be a ringed site. Then a \textit{vector bundle} on $\XX$ is a locally free $\OO_\XX$-module of finite and constant rank.\end{defi}

Our first main theorem concerns the situation at primes bigger than $3$:

\begin{thmA}\label{TheoremA}Let $R$ be a discrete valuation ring or a field with (residue) characteristic not equal to $2$ or $3$. Then every vector bundle on $\MM_R$ is a sum of line bundles.
\end{thmA}
The author learned of this result in the case of a field from Angelo Vistoli, where it is also true for the compactified moduli stack. Theorem A will be proven in the more general context of weighted projective and affine lines in Section 3. As the question whether every vector bundle on $\A^1_R$ for a regular local ring decomposes into (trivial) line bundles is the open Bass--Quillen conjecture (see \cite[VIII.6]{Lam06}), an extension of Theorem A to arbitrary regular local rings does not seem feasible in the moment, though some generalizations are certainly possible. As our applications to stable homotopy theory (see \cite{Mei12}) only require the case $R= \Z_{(p)}$, we stay with discrete valuation rings. Note also that the results in \cite{H-S99b} indicate that Theorem A might not be true for the compactified moduli stack if $R$ is not a field. \\

The rest of this article will be concerned with the situation at $2$ and $3$, where not every vector bundle splits into line bundles. The easiest examples of indecomposable higher rank vector bundles are non-trivial extensions of line bundles. At the prime $3$ we achieve a classification of all vector bundles which are iterated extensions of line bundles: 
\begin{thmB}
Let $\EE$ be a vector bundle on $\MM_{(3)} = \MM_{\Z_{(3)}}$ that is an iterated extension of line bundles. Then $\EE$ is a sum of vector bundles of the form $\omega^k$, $\omega^k\tensor E_\alpha$ and $\omega^k\tensor f_*f^*\OO$. 
\end{thmB}
Here, the vector bundle $E_\alpha$ is of rank $2$ and will be introduced in Subsection \ref{ExIndec}. The vector bundle $f_*f^*\OO$ is of rank $3$ and arises from the forgetful map $f\co\MM_0(2)_{(3)} \to \MM_{(3)}$ from the moduli stack of elliptic curves with one chosen point of exact order $2$.

The situation at the prime $2$ seems to be more involved and here we can prove that there are actually infinitely many indecomposable vector bundles:

\begin{thmC}There are infinitely many indecomposable vector bundles on $\MM_{(2)} = \MM_{\Z_{(2)}}$ (of arbitrary high rank).\end{thmC}

This indicates that a complete classification of vector bundles on $\MM_{(2)}$ might be quite difficult, while the classification of vector bundles on $\MM_{(3)}$ could be achieved by a positive answer to the following question:
\begin{question}Is every vector bundle on $\MM_{(3)}$ an iterated extension of line bundles?\end{question}

At last, we give an overview over the structure of this article. The second section provides some basics about weighted projective stacks. The third section proves Theorem A, first for fields and then for discrete valuation rings. This will be done in the more general context of weighted projective and affine lines. In Section \ref{examples} we will give a detailed (cohomological) analysis of some low-rank vector bundles on $\MM_{(3)}$ and their extensions. In Section \ref{Classifying} we will exploit this to show that we have already constructed all iterated extensions of line bundles, i.e.\ we prove Theorem B. In Section \ref{vec2}, we see how integral representations of $GL_2(\F_3)$ give rise to vector bundles on $\MM_{(2)}$ and use representation-theoretic arguments to prove Theorem C. In the appendix, we will review some results about quasi-coherent sheaves and completions used in Section \ref{VBA23}.  

\subsection*{Acknowledgements}
Significant parts of this article are taken from my thesis \textit{United Elliptic Homology} written under the supervision of Stefan Schwede -- to him belongs my gratitude for his guidance. I also want to thank Angelo Vistoli, Jochen Heinloth and Peter Scholze for helpful e-mail exchanges, Akhil Mathew and Martin Brandenburg for their mathoverflow answers, the referee for his or her suggestions and Viktoriya Ozornova for reading preliminary versions. Further thanks belong to the GRK 1150 and the Deutsche Telekom Stiftung for monetary support. 

\section{Basics About Weighted Projective Stacks}\label{ProjectiveLines}
Away from $2$ and $3$, the moduli stack of elliptic curves is a weighted projective line and it is more natural to study vector bundles in this context. For the convenience of the reader, we define in this section weighted projective stacks, prove some of their properties and compute their cohomology. This is probably all well-known and we claim no originality here. 

For $a_0,\dots, a_n \in \N$ and a commutative ring $R$, the \textit{weighted projective stack} $\PP_R(a_0,\dots, a_n)$ is the (stack) quotient of $\A_R^n-\{0\}$ by the multiplicative group $\G_m$ under the action which is the restriction of the map
\begin{eqnarray*} \phi\co \G_m\times \A_R^{n+1} &\to& \A_R^{n+1} \\
 \Z[t, t^{-1}]\otimes R[t_0,\dots, t_n] &\leftarrow& R[t_0,\dots, t_n] \\
 t^{a_i}\tensor t_i &\mapsfrom& t_i
\end{eqnarray*}
to $\G_m\times (\A^{n+1}_R-\{0\})$. Here, $\A_R^{n+1}-\{0\}$ denotes the complement of the zero point, i.e.\ of the common vanishing locus of all $t_i$. On geometric points, the action corresponds to the map $(t, t_0,\dots, t_n) \mapsto (t^{a_0}t_0,\dots, t^{a_n}t_n)$. In the special case of $n=1$ we speak of a \textit{weighted projective line}.

Recall that there is an equivalence between affine schemes with $\G_m$-action and graded rings given as follows: If $\Spec A$ has a $\G_m$-action, an element $a \in A$ is homogeneous of degree $i$ if and only if $a\mapsto a \tensor t^i$ under the map $A \to A \tensor \Z[t,t^{-1}]$ corresponding to the action map $\Spec A \times \G_m \to \Spec A$. Under this correspondence, $\phi$ corresponds thus to the grading $|t_i| = a_i$. 

I learned the following proof from Akhil Mathew. 
\begin{prop}\label{weightedproper}For every $a_0,\dots, a_n \in \N$ and every commutative ring $R$, the projective stack $X = \PP_R(a_0,\dots, a_n)$ is a proper and smooth Artin stack over $\Spec R$.\end{prop}
\begin{proof}Since $X \simeq \PP_\Z(a_0,\dots, a_n) \times_{\Spec \Z} \Spec R$ and properness is preserved under base change, we can assume $R = \Z$. Every stack quotient of a smooth scheme by a smooth group scheme is a smooth Artin stack, so $X$ is a smooth Artin stack.

Next, we have to show that $X$ is separated. For a preliminary remark assume that $A$ is a local ring. Then the groupoid $X(A)$ is equivalent to the groupoid with objects $(n+1)$-tuples $(x_0,\dots, x_n) \in A^{n+1}$ where at least one $x_i$ is invertible and a morphism between $(x_0,\dots, x_n)$ and $(y_0,\dots, y_n)$ is an element $\lambda \in A^\times$ such that $(\lambda^{a_0}x_0, \dots, \lambda^{a_n}x_n) = (y_0,\dots, y_n)$. Indeed, while $X$ is the stack of $\G_m$-torsors with $\G_m$-equivariant map to $\A^{n+1}-\{0\}$, we just described the value of the prestack of \textit{trivial} $\G_m$-torsors with $\G_m$-equivariant map to $\A^{n+1}-\{0\}$ (see \cite{Beh12}, Definition 3.11 and Example 3.12). These agree on every scheme whose Picard group is trivial. 

For $\pi\co X \to \Spec \Z$ to be separated it suffices to show that for every valuation ring $A$ with quotient field $K$ and every two objects $x,y \in X(R)$ with an isomorphism $\overline{\lambda}\co x_K \to y_K$ in $X(K)$, there is exactly one isomorphism $\lambda\co x \to y$ in $X(A)$ such that $\lambda_K = \overline{\lambda}$ (\cite[Proposition 7.8]{LMB00}). As a valuation ring is local, $x$ corresponds to $(x_0,\dots, x_n)\in A^{n+1}$, $y$ to $(y_0,\dots, y_n) \in A^{n+1}$ and $\overline{\lambda}$ to an element $\lambda'\in K^\times$ with $\lambda'^{a_i}x_i = y_i$. If $x_i$ is invertible, $\lambda'^{a_i} \in A$ and thus also $\lambda' \in A$.  

As $\pi$ is thus separated and also of finite type, for $\pi$ to be proper it is now enough to show the following: For every discrete valuation ring $A$ with quotient field $K$ and object $x\in X(K)$, there is a finite extension $L$ of $K$ such that we have an object $y\in X(B)$ (for $B$ the integral closure of $A$ in $L$) and an equivalence $y_L \to x_L$ in $X(L)$ (see \cite[Theoreme 7.10]{LMB00}). The object $x$ corresponds again to an $(n+1)$-tuple $(x_0,\dots, x_n) \in K^{n+1}$. The object $x$ is isomorphic to $y_K$ for some $y \in X(A)$ if and only if the valuation of $x_i$ is divisible by $a_i$ for $i= 0,\dots, n+1$. This can be assumed after ramified (finite) base change. 
\end{proof}

\begin{remark}\label{genericautomorphism}
Set again $X = \PP_R(a_0,\dots, a_n)$ and $x \in X(k)$ be a field-valued point corresponding to an $(n+1)$-tuple $(x_0,\dots, x_n) \in k^{n+1}$. We claim that the automorphism group scheme $\underline{Aut}(x)$ is isomorphic to $\mu_d = \Spec k[u]/(u^d-1)$, where $d$ is the greatest common divisor of all $a_i$ such that $x_i\neq 0$.

Indeed, let $R$ be a $k$-algebra and denote by $x_R\in X(R)$ the pullback. As in the last proof, the automorphism group of $x_R$ consists of all $\lambda \in R$ with 
\[(\lambda^{a_0}x_0, \dots, \lambda^{a_n}x_n) = (x_0,\dots, x_n).\]
As the $x_i$ are either invertible or zero, the condition is exactly $\lambda^{a_i} = 1$ if $x_i \neq 0$. Thus, we have a natural bijection $\underline{Aut}(x)(R) \cong \mu_d(R)$. 

 A stack quotient of a smooth scheme by a smooth group scheme is a smooth Deligne--Mumford stack if the stabilizers are finite and reduced (see \cite{Edi00}). We see that $\PP_R(a_0,\dots, a_n)$ is actually a \textit{Deligne--Mumford stack} if all $a_i$ are invertible on $R$. 
\end{remark}

\begin{prop}\label{FinitePn}For every $a_0,\dots, a_n \in \N$ and every commutative ring $R$, there is a finite fpqc map $\psi\co \mathbb{P}_R^n = \PP_R(1,\dots, 1) \to \PP_R(a_0,\dots, a_n)$. \end{prop}
\begin{proof}By definition, we have $\mathbb{P}_R^n = (\Spec R[x_0,\dots, x_n] -\{0\})//\G_m$ and $\PP_R(a_0,\dots, a_n) = (\Spec R[t_0,\dots, t_n] -\{0\})//\G_m$ with gradings $|x_i| = 1$ and $|t_i| = a_i$ for all $i=0,\dots, n$. We define a map $f\co R[t_0,\dots, t_n] \to R[x_0,\dots, x_n]$ of graded rings by $t_i \mapsto x_i^{a_i}$. This map makes $R[x_0,\dots, x_n]$ into a free module of finite rank over $R[t_0,\dots, t_n]$. The map $f$ induces a $\G_m$-equivariant map $g\co \Spec R[x_0,\dots, x_n] \to\Spec R[t_0,\dots, t_n]$, which is finite, flat, quasi-compact and surjective. As all these properties descent via fpqc-maps and 
\[\xymatrix{\Spec R[x_0,\dots, x_n] \ar[r]\ar[d] & \Spec R[t_0,\dots, x_n]\ar[d] \\
\Spec R[x_0,\dots, x_n]//\G_m \ar[r] & \Spec R[t_0,\dots, t_n] //\G_m } \]
is a pullback square, the map $g//\G_m\co \Spec R[x_0,\dots, x_n]//\G_m \to \Spec R[t_0,\dots, t_n] //\G_m$ has all these properties as well. As the complement of the vanishing locus of $x_0,\dots, x_n$ in the same as that of $x_0^{a_0},\dots, x_n^{a_n}$, this restricts to a finite fpqc map
\[\psi\co \mathbb{P}_R^n \to \PP_R(a_0,\dots, a_n).\]
\end{proof}

For every integer $m$, there is a line bundle $\OO(m)$ on $\PP_R(a_0,\dots, a_n)$ defined as follows: The map $\phi$ gives a $\G_m$-action on $\A^n_R$. The category of quasi-coherent modules on $\A^n_R//\G_m$ is equivalent to graded $R[t_0,\dots, t_n]$-modules by Galois descent, where $|t_i| = a_i$. For $M$ a graded module, denote by $M[m]$ the graded module with $M[m]_k = M_{m+k}$. Then $R[t_0,\dots, t_n][m]$ is a graded $R[t_0,\dots, t_n]$-module, which corresponds to a line bundle on $\A^n_R//\G_m$ whose restriction to $\PP(a_0,\dots, a_n)$ we denote by $\OO(m)$. 

\begin{example}\label{mmb}Denote by $\MMb_R$ the compactified moduli stack of elliptic curves over a ring $R$ (in the sense of $\mathfrak{M}_1$ in \cite[IV.2.4]{D-R73}). We have an equivalence \[\MMb_R \simeq (\Spec R[c_4,c_6]-\{0\}) // \G_m \cong \PP(4,6)\] for $R$ a ring where $6$ is invertible. This equivalence is given by the Weierstraß form $E\co y^2 = x^3 +c_4x + c_6$. Under this equivalence, the line bundle $\omega$ on $\MMb_R$ corresponds to $\OO(1)$ on $\PP_R(4,6)$. The reason is that $f(u)^*\omega_0 = u \omega_0$, where $\omega_0 = \frac{dx}{2y}$ is the usual invariant differential, which is a local trivialization of $\omega$, and $f(u)$ is the endomorphism of the elliptic curve $E$ given by $f(u)(x) = u^{-2}x$ and $f(u)(y) =u^{-3}y$. 
\end{example}

Our next aim is to calculate the cohomology of the line bundle $\OO(m)$ on a weighted projective stack.

\begin{prop}\label{CohomologyProjective}
Let $X = \PP_R(a_0,\dots, a_n)$ be a weighted projective stack. Define sets 
\begin{align*}
A(m) &= \{(\lambda_0,\dots, \lambda_n) \in \Z_{\geq 0}^{n+1}:\; \sum_{i=0}^n \lambda_i a_i = m\},\text{ and} \\
B(m) &= \{(\lambda_0,\dots, \lambda_n) \in \Z_{<0}^{n+1}:\; \sum_{i=0}^n \lambda_i a_i = m\}.
\end{align*}
Then 
\begin{align*}
H^0(X;\OO(m)) &= \text{free }R\text{-module on }A(m) \\
H^i(X; \OO(m)) &= 0\; \text{ for } 1\leq i \leq n-1 \\
H^n(X;\OO(m)) &= \text{free }R\text{-module on }B(m)
\end{align*}
Note, in particular, that $H^0(X;\OO(m)) = 0$ for $m<0$ and 
\[H^0(X;\OO(m)) \cong H^n\left(X; \OO(-\sum_{i=0}^n a_i - m)\right).\] 
\end{prop}
\begin{proof}
We compute the cohomology as \v{C}ech cohomology with respect to the cover by the open substacks 
\[U_i = \Spec R[t_0,\dots, t_n][t_i^{-1}]//\G_m.\]
 Note that $H^j(U_i;\FF)= 0$ for $j>0$ and every quasi-coherent sheaf $\FF$ by Lemma \ref{gradetriv}. 

The \v{C}ech complex associated to this cover and the sheaf $\bigoplus_{l\in\Z}\OO(l)$ agrees with the one for the standard covering of the projective space $\mathbb{P}^n_R$, if we view $\bigoplus_{l\in\Z}\OO(l)$ as an ungraded sheaf. As in this classical case (see \cite[III.5.1]{Har77}), we obtain
\begin{align*}
H^0(X;\bigoplus_{l\in\Z}\OO(l)) &= R[t_0, \dots, t_n] \\
H^i(X; \bigoplus_{l\in\Z}\OO(l)) &= 0\; \text{ for } 1\leq i \leq n-1 \\
H^n(X;\bigoplus_{l\in\Z}\OO(l)) &= \begin{minipage}{12cm}free $R$-module on the set of 
monomials $t_0^{i_0}\cdots t_n^{i_n}$ with all $i_j<0$ \end{minipage}
\end{align*} 
The group $H^i(X;\OO(m))$ is isomorphic to the part of degree $m$, where $|t_i| = a_i$ again. The proposition follows easily now.
\end{proof}

As an example, we depict some range of cohomology in the case of $\PP_R(4,6)$, where squares stand for copies of $R$:

\[\begin{sseq}[ylabelstep=1]{-22...12}{2}
 \ssdrop{\Box}\ssmove{4}{0}
\ssdrop{\Box}\ssmove{2}{0}
\ssdrop{\Box}\ssmove{2}{0}
\ssdrop{\Box}\ssmove{2}{0}
\ssdrop{\Box}\ssmove{2}{0}
\ssdrop{\Box \Box}
\ssmoveto{-22}{1}
\ssdrop{\Box\Box}\ssmove{2}{0}
\ssdrop{\Box}\ssmove{2}{0}
\ssdrop{\Box}\ssmove{2}{0}
\ssdrop{\Box}\ssmove{2}{0}
\ssdrop{\Box}\ssmove{4}{0}
\ssdrop{\Box}
\end{sseq}\]

\section{Vector Bundles Away From $2$ and $3$}\label{VBA23}
\subsection{The Case of a Field}
I have learned most of the proofs in this subsection from Angelo Vistoli. 

Recall that a coherent sheaf $\FF$ is called \textit{reflexive} if the canonical map from $\FF$ to its double-dual is an isomorphism. In particular, every vector bundle is reflexive. 

\begin{lemma}\label{PBReflx}Pullbacks along flat maps preserve reflexive sheaves. \end{lemma}
\begin{proof}It follows from \cite[IV, Proposition 18]{Ser00} that pullbacks along flat maps between affine schemes preserve duals. Since flatness is a local condition and duals are computed locally, thus pullback along any flat map between algebraic stacks preserves duals.\end{proof}

\begin{lemma}\label{Reflextension}Let $\XX$ be an Artin stack and $i\co U\to \XX$ a quasi-compact open immersion. For every reflexive sheaf $\xi$ on $U$, there is a reflexive sheaf $\FF$ on $\XX$ with $i^*\GG \cong \xi$. \end{lemma}
\begin{proof}Note first that $i_*\xi$ is quasi-coherent by \cite[Proposition 13.2.6]{LMB00} since $i$ is quasi-compact. By \cite[Corollaire 15.5]{LMB00} there is then a coherent sheaf $\GG$ on $\MMb_R$ with $i^*\GG = \xi$. Let $\FF$ denote its double-dual. This is reflexive (\cite{Har80}, 1.2 - which we can use also for stacks since both reflexivity and coherence are local conditions) and, in addition, we have $i^*\FF = \xi$ since $\xi$ is already reflexive and $i$ is flat. \end{proof}
By the lemma, we get directly the following: 
\begin{prop}For $R$ a ring, every reflexive sheaf on $\MM_R$ is the restriction of a reflexive sheaf on $\MMb_R$. 
\end{prop}

By Example \ref{mmb}, we have an equivalence $\MMb_K \simeq (\Spec K[c_4,c_6]-\{0\}) // \G_m = \PP_K(4,6)$ for $K$ a a field of characteristic not $2$ or $3$. Therefore, it is enough to classify reflexive sheaves on weighted projective lines. 

\begin{prop}\label{ClassField}Let $K$ be an arbitrary field, $m,n\in\N$. Then every reflexive sheaf $\FF$ on $\PP_K(m,n)$ is a direct sum of line bundles of the form $\OO(a)$.\end{prop}
\begin{proof}By Galois descent, the sheaf $\FF$ corresponds to a $\G_m$-equivariant sheaf on $\A_K^2-\{0\}$, with respect to the action given by $t(x,y) = (t^m x, t^n y)$; we will denote this $\G_m$-equivariant sheaf by abuse of notation still by $\FF$. This new sheaf $\FF$ is reflexive since pullbacks by flat maps preserve reflexive sheaves by Lemma \ref{PBReflx}. Using the inclusion $(\A_K^2-\{0\})//\G_m \hookrightarrow \A_K^2//\G_m$, we can by Lemma \ref{Reflextension} extend $\FF$ to a reflexive $\G_m$-equivariant sheaf on $\A_K^2$, which we denote by abuse of notation again by $\FF$. Since every reflexive sheaf on a regular $2$-dimensional scheme is locally free (\cite[Corollary 1.4]{Har80}), $\FF$ is locally free as a non-equivariant sheaf. The datum of a $\G_m$-equivariant sheaf on $\A_K^2$ is equivalent to that of a graded module over $K[t_1,t_2]$. Each shift $K[t_1,t_2][a]$ for $a\in\Z$ defines a $\G_m$-equivariant line bundle on $\A^2$, restricting to $\OO(a)$ on $\PP_K(m,n)$. Thus, our proposition follows directly from the following more general (and not very difficult) proposition.\end{proof}

\begin{prop}[\cite{Lam06}, II.4.6]Let $R$ be a ring graded in degrees $\geq 0$, $R_0$ its zeroth degree part and $R^+$ its part of positive degree. Then for every graded projective $R$-module, we have an isomorphism of graded $R$-modules $R\tensor_{R_0} (P/R^+P) \cong P$. In particular, if $R_0$ is a field, every graded $R_0$-module is a sum of shifts of $R_0$ and so $P$ is isomorphic to a sum of shifts of $R$.\end{prop}

\begin{cor}For $K$ a field of characteristic not $2$ or $3$, every vector bundle on $\MM_K$ or $\MMb_K$ is a direct sum of line bundles.\end{cor}

\begin{remark}
 The analogue of Proposition \ref{ClassField} is certainly not true for higher dimensional (weighted) projective spaces. There is a nice overview of the topic of indecomposable vector bundles on projective spaces in the introduction of \cite{M-Z05}.
\end{remark}

\subsection{Vector Bundles on Weighted Affine Lines over a Discrete Valuation Ring}
We do not attempt to classify vector bundles over weighted projective lines over a regular local ring; this is probably difficult even if $\dim(R) = 1$ as in \cite{H-S99b}. We try to classify vector bundles over weighted \textit{affine} lines instead. 
\begin{defi}Let $R$ be a ring and $a,b$ be positive integers. Then a \textit{weighted affine line of type }$(a,b)$ is the non-vanishing locus of a section $\Delta \in \OO(m)(\PP_R(a,b))$ whose vanishing locus is isomorphic to $\Spec R//\mu_{\gcd(a,b)}$.\end{defi}
An example of a weighted affine line is $\MM_R$ if $\frac16\in R$. Indeed, it is the non-vanishing locus of $\Delta \in \OO(12)(\PP_R(4,6))$ where $\Delta = \frac{c_4^3-c_6^2}{1728}$ for $\PP_R(4,6) = (\Spec R[c_4,c_6] -\{0\})//\G_m$. 

We wish to prove the following theorem:
\begin{thm}\label{WPV}Let $a$ and $b$ be positive integers and $R$ be a discrete valuation ring. Then every vector bundle on a weighted affine line of type $(a,b)$ is a direct sum of line bundles of the form $\OO(n)$ (or rather their restrictions). \end{thm}
\begin{cor}Every vector bundle on $\MM_R$ for $R$ a discrete valuation ring with $\frac16\in R$ is a sum of tensor powers of $\omega$. In particular, this holds for $R = \Z_{(p)}$ for $p>3$. \end{cor}

We will mimick the proof strategy of \cite{Hor64}, where Horrocks proves the analogous result for non-weighted affine lines. More precisely, we will first show that a vector bundle on a weighted affine line is trivial if it can be extended to the corresponding weighted projective line (for an arbitrary noetherian local ring $R$) and then that vector bundles can be extended (for $R$ a regular local ring of dimension $\leq 1$).

\begin{lemma}\label{Vectornot}Let $R$ be a local ring and $\XX$ be an algebraic stack over $R$ such that $\XX \to \Spec R$ is a closed morphism. Assume furthermore that there is a closed fpqc cover $Y \to \XX$ for some reduced scheme $Y$. Let 
    \[0\to  \EE \to \FF \to \GG \to 0\]
be a short exact sequence of sheaves on $\XX$ where $\EE$ and $\FF$ are vector bundles. Assume that 
\[0\to \EE\tensor_Rk \to \FF\tensor_Rk \to \GG\tensor_Rk \to 0\]
is a short exact sequence of vector bundles on $\XX\times_{\Spec R} \Spec k$ for $R/\m_R = k$ the residue field. Then $\GG$ is a vector bundle.
\end{lemma}
\begin{proof}Let $A$ be the local ring of an arbitrary closed point in $Y$. It is enough to show that $\GG(\Spec A)$ is a free $A$-module. We have exact sequences
\[0\to  \EE(\Spec A) \to \FF(\Spec A) \to \GG(\Spec A) \to 0\]
and
   \[0\to  \EE(\Spec A)\tensor_Rk \to \FF(\Spec A)\tensor_R k \to \GG(\Spec A)\tensor_R k \to 0.\]
There is an isomorphism $\GG(\Spec A)\tensor_R k \cong (A\tensor_R k)^{\rk\FF - \rk \EE}$. Lifting the basis elements, we get an $A$-linear map $f\co A^{\rk\FF-\rk \EE} \to \GG(\Spec A)$. It is surjective, as can be tested at the closed point of $\Spec A$, which maps to the closed point of $\Spec R$. 

The ring $A$ is reduced, hence integral. Denote by $K$ the quotient field of $A$. Then 
\[f_K\co A^{\rk\FF-\rk \EE}\tensor_A K \to \GG(\Spec A)\tensor_A K\]
 is also a surjection; hence an isomorphism since we have a short exact sequence
\[0 \to \EE(\Spec K) \to \FF(\Spec K) \to \GG(\Spec K) = \GG(\Spec A)\tensor_A K \to 0,\]
which implies that $\GG(\Spec A)\tensor_A K \cong A^{\rk\FF-\rk \EE}\tensor_A K$. Since $K$ is a flat $A$-module, we obtain $\ker(f) \tensor_A K = 0$ and $\ker(f)$ is completely $A$-torsion. But $A^{\rk \FF - \rk \EE}$ is $A$-torsionfree. Therefore, $\ker(f) = 0$ and $f$ is an isomorphism. 
\end{proof}

\begin{prop}Let $R$ be a noetherian reduced local ring. Let $\EE$ be vector bundle on $X = \PP_R(a,b)$. Then there exists a short exact sequence
\[0 \to \OO(l) \to \EE \to \EE/\OO(l)\to 0\]
with $\EE/\OO(l)$ a vector bundle on $X$ and $l\in\Z$. \end{prop}
\begin{proof}
 Write $\m$ for the maximal ideal of $R$ and $k$ for $R/\m$. The sheaf $\EE_k = \EE\tensor_R k$ is a locally free sheaf on $X_k = X \times_{\Spec R}\Spec k$. Denote by $\OO(n)_k$ the sheaf $\OO(n)$ on $X_k$. By Proposition \ref{ClassField}, we have an isomorphism $\EE_k\cong \bigoplus_n r_n \OO(n)_k$. By tensoring $\EE$ with a suitable power of $\OO(1)$, we can assume that $r_n= 0$ for $n<0$ and $r_0 \neq 0$. 

Denote the completion of $R$ with respect to $\m$ by $\widehat{R}$. Since $\widehat{R}$ is flat over $R$, we have $\Gamma(\EE)\tensor_R \widehat{R} \cong \Gamma(\EE\tensor_R \widehat{R})$. By Theorem 11.ii from \cite{Ols07}, we get thus an isomorphism
\[\Gamma(\EE)\tensor_R \widehat{R} \xrightarrow{\cong} \varprojlim_i \Gamma(\EE\tensor_R R/\m^i).\]

We are going to show that the homomorphisms in the inverse limit are surjective. Since $\EE$ is locally free, the sequence
\[0\to \EE \tensor_R \m^{i-1}/\m^i \to \EE \tensor_R R/\m^i \to \EE \tensor_R R/\m^{i-1} \to 0\]
is exact for $i\geq 2$. The $R$-module $\m^{i-1}/\m^i$ is isomorphic to $k^N$ for some $N$ for $i\geq 1$. So
\[H^1(X; \EE\tensor_R \m^{i-1}/\m^i) \cong \bigoplus^N H^1(X_k;\EE_k) \cong \bigoplus^N \bigoplus_n r_n H^1(X_k; \OO(n)_k).\]
By Proposition \ref{CohomologyProjective}, $H^1(X;\OO(n)_k)$ vanishes for $n\geq 0$. Thus, 
\[\Gamma(\EE\tensor_R R/\m^i) \to \Gamma(\EE\tensor_R R/\m^{i-1})\]
 is surjective for $i\geq 2$ and thus also 
\[\Gamma(\EE)\tensor_R \widehat{R} \cong \varprojlim_i \Gamma(\EE\tensor_R R/\m^i) \to \Gamma(\EE_k).\]
As the target is a $k$-module, the morphism factors over 
\[\Gamma(\EE) \tensor_R k \to \Gamma(\EE_k)\]
and thus also
\[\Gamma(\EE) \to \Gamma(\EE)\tensor_R k \to \Gamma(\EE_k)\] is surjective. 

Since $r_0 \neq 0$, there is an injective map $\overline{f}\co \OO_{X_k} \to \EE_k$ whose cokernel is a vector bundle of rank $\rk \EE -1$. The map $\overline{f}$ corresponds to a nowhere vanishing section $\overline{s} \in \Gamma(\EE_k)$ and since $\Gamma(\EE) \to \Gamma(\EE_k)$ is surjective, we can lift $\overline{s}$ to a section $s\in\Gamma(\EE)$; since every closed point of $X$ is over the closed point of $\Spec R$, this is also nowhere vanishing. The section $s$ corresponds to an injective map $\OO_X \xrightarrow{f} \EE$, yielding a short exact sequence
\[0 \to \OO_X \xrightarrow{f} \EE \to \EE/\OO_X \to 0.\]
By Proposition \ref{weightedproper} and Proposition \ref{FinitePn}, $\PP_R(a,b)$ is a proper Artin stack and $\mathbb{P}_R^1\to \PP_R(a,b)$ a proper fpqc map from a reduced scheme. As proper maps are closed, Lemma \ref{Vectornot} implies that $\EE/\OO_X$ is a vector bundle again.  
\end{proof}

\begin{cor}\label{VectorCor}Let $R$ be a noetherian reduced local ring. Let $\EE$ be vector bundle on $X = \PP_R(a,b)$ and $j\co U \hookrightarrow X$ be an open substack such that $\Ext^1(\FF; \GG)$ vanishes for all vector bundles $\FF$ and $\GG$. Then $j^*\EE$ is a direct sum of line bundles.\end{cor}
\begin{proof}
 By the last proposition, there is a short exact sequence of vector bundles
\[0\to \OO_U(k) \to j^*\EE \to j^*\EE/j^*\OO(k)) \to 0. \]
As $\Ext^1(j^*\EE/j^*\OO_U(k), j^*\OO(k)) = 0$, the extension splits. The assertion follows by induction.
\end{proof}

\begin{prop}Let $R$ be a discrete valuation ring and $U \subset \PP_R(a,b)$ be a weighted affine line with $\gcd(a,b) = 1$. Then every vector bundle $\EE$ on $U$ is the restriction of a vector bundle $\FF$ on $\PP_R(a,b)$.\end{prop}
\begin{proof}Let $U$ be the non-vanishing locus of some $\Delta \in \OO(m)(\PP_R(a,b))$. Then the completion of $\PP_R(a,b)$ at the vanishing locus of $\Delta$ is equivalent to $\Spf R[[q]]$; here, $q$ corresponds to $\Delta$ under a trivialization of $\OO(m)$ on $\Spf R[[q]]$. Let $\hat{U}$ be the ''non-vanishing locus'' of $\Delta$ (or, equivalently, $q$) on this completion or, more precisely, the ringed site with same underlying site as $\PP_R(a,b)$ and sheaf of rings $\OO_{\hat{U}} = \widehat{\OO_{\PP_R(a,b)}}[\frac1\Delta]$ as in Theorem \ref{glueing}. 

By Theorem \ref{glueing} and Lemma \ref{AffineEquivalence}, the datum of a vector bundle on $\PP_R(a,b)$ is thus equivalent to that of a vector bundle $\EE$ on $U$, a finitely generated projective module $P$ on $R[[q]]$ and an isomorphism $\EE(\hat{U}) \to P[\frac1q]$. Hence, a vector bundle $\EE$ on $U$ is the restriction of a vector bundle on $\PP_R(a,b)$ if there is a finitely generated projective $R[[q]]$-module $P$ such that $\EE(\hat{U}) \cong P[\frac1q]$. But by \cite[Theorem 2]{Hor64} every projective module over $R((q))= R[[q]][\frac1q]$ is free, in particular $\EE(\hat{U})$. Thus, we can choose $P$ to be $R[[q]]^{\rk \EE}$. 
\end{proof}

 \begin{lemma}\label{gradetriv}Let $\XX \simeq \Spec A //\G_m$ for a graded ring $A$. Then $\Ext^i_{\OO_\XX}(\EE,\FF) = 0$ for $\EE$ a vector bundle, $\FF$ a quasi-coherent sheaf and $i>0$.\end{lemma}
 \begin{proof}Quasi-coherent sheaves on $\XX$ are equivalent to graded $A$-modules by Galois descent. Clearly, the functor $\Hom_{A\grmod}(A,-)$ is exact on graded $A$-modules. Thus, $H^i(\XX; \FF) = 0$ for every $i>0$ and every quasi-coherent sheaf $\FF$. But, $\Ext^i_{\OO_\XX}(\EE,\FF) \cong H^i(\XX; \mathcal{H}om_{\OO_\XX}(\EE,\FF))$ by the degenerated Grothendieck spectral sequence, where $\mathcal{H}om$ denotes the Hom-sheaf, since the Ext-sheaves $\mathcal{E}xt_{\OO_\XX}^i(\EE,\FF)$ vanish for $i>0$.\end{proof}

\begin{proof}[Proof of Theorem \ref{WPV}: ]
 We want to reduce to the case $\gcd(a,b) = 1$. For this reduction it is enough to show that if the theorem is true for $(a,b)$, then it is also true for $(la, lb)$ for $l\in\Z_{>0}$. Let $\PP_R(a,b) = (\Spec R[x_a, x_b]-\{0\})//\G_m$ and $U$ be a weighted projective line, i.e.\ the non-vanishing locus of some section $\Delta \in \Gamma(\OO(m))$. As $\Delta$ corresponds to a homogeneous polynomial of degree $>0$, $U$ can also be seen as the non-vanishing locus of $\Delta$ on $\Spec R[x_a,x_b]//\G_m$. Thus, $U\simeq \Spec R[x_a,x_b,\Delta^{-1}]//\G_m$ and $\QCoh(U) \simeq \Spec R[x_a,x_b,\Delta^{-1}]\grmodules$. The category $R[x_{la},x_{lb},\widetilde{\Delta}^{-1}]\grmodules$ (where $|x_{la}| = la$, $|x_{lb}| = lb$ and $|\widetilde{\Delta}| = l|\Delta|$) is equivalent to an $l$-fold product of $R[x_a,x_b,\Delta^{-1}]\grmodules$. The equivalence is given by sending a tuple $(M_0,\dots, M_{l-1})$ of graded $R[x_a,x_b,\Delta^{-1}]$-modules to $\bigoplus_{i=0}^{l-1} M_i[i]$. Reformulating, we get an equivalence from the $l$-fold product of $\Vect(\PP_R(a,b)-V(\Delta))$ to $\Vect(\PP_R(la,lb) - V(\widetilde{\Delta}))$ sending a tuple $(\EE_0,\dots, \EE_{l-1})$ to $\bigoplus_{i=0}^{l-1} \EE_i\tensor \OO(i)$. If the $\EE_i$ decompose all into line bundles, also $\bigoplus_{i=0}^{l-1} \EE_i\tensor \OO(i)$ does. 

Assume now $\gcd(a,b) =1$ and let $U\subset \PP_R(a,b)$ be a weighted affine line with a vector bundle $\EE$ on it. We can use the last proposition to find a vector bundle $\FF$ on $\PP_R(a,b)$ which restricts to $\EE$. Thus, $\EE$ is a sum of line bundles by Corollary \ref{VectorCor} as $U \simeq \Spec R[x_a, x_b, \Delta^{-1}]//\G_m$ and thus all Ext-groups between vector bundles vanish by Lemma \ref{gradetriv}.
\end{proof}

\section{Examples of Vector Bundles on $\MM_{(3)}$ and Their Extensions}\label{examples}
In this section, we will give a detailed exposition of some vector bundles of low rank on the moduli stack of elliptic curves at $p=3$. Our main aim will be to compute the groups of extensions between them, which will be key to the classification of iterated extensions of line bundles in the next section. 

To compute these Ext-groups we will first recall the cohomology of the moduli stack of elliptic curves. Then we will discuss two examples of indecomposable vector bundles, $E_\alpha$ and $f_*f^*\OO$. In the third subsection, we will show that $f_*f^*\OO$ is self-dual and construct two short exact sequences connecting $E_\alpha$ and $f_*f^*\OO$ by representation theoretic methods. In the fourth subsection, we will classify extensions between $E_\alpha$ and $E_\alpha \tensor \omega^j$ for arbitrary $j$.

\subsection{The Cohomology of the Moduli Stack of Elliptic Curves}\label{EllCoh}
A computation of the cohomology on $\MM_{(3)} = \MM_{\Z_{(3)}}$ can be found in slightly different language in \cite[Section 3]{Bau08}. More precisely, he calculates the cohomology of the graded Hopf algebroid \[(A = \Z_{(3)}[a_1,a_2,a_3,a_4,a_6], A[r,s,t]),\] 
where $|a_i| = i$, $|r| = 2$, $|s| = 1$ and $|t| = 3$. 
Recall that a Hopf algebroid is just a cogroupoid object in commutative rings (we give here just the rings corepresenting objects and morphisms of this groupoid and leave the structure maps implicit). A graded Hopf algebroid a cogroupoid object in (strictly) commutative graded rings. Given an algebraic stack $\XX$ and an affine morphism $q\co \Spec R \to \XX$, we get a Hopf algebroid $(R,\Gamma)$, where 
\[\Spec R\times_\XX \Spec R \simeq \Spec \Gamma.\] If $q$ is faithfully flat, we get an equivalence of categories between quasi-coherent sheaves on $\XX$ and $(R,\Gamma)$-comodules (see \cite[Section 3.4]{Nau07}). 

In the case of the moduli stack of elliptic curves, the Weierstraß form gives a fpqc map $\Spec A[\Delta^{-1}] \to \MM_{(3)}$ with 
\[\Spec A[\Delta^{-1}] \times_{\MM_{(3)}} \Spec A[\Delta^{-1}] \simeq \Spec A[\Delta^{-1}][r,s,t,u^{\pm 1}].\] 
Comodules over 
\[(A[\Delta^{-1}], A[r,s,t,u^{\pm 1}][\Delta^{-1}])\]
 are equivalent to graded comodules over 
\[(A[\Delta^{-1}], A[r,s,t][\Delta^{-1}]).\] One can check that the structure maps of this Hopf algebroid agree with the those found in \cite[Section 3]{Bau08}.

The line bundle $\omega$ on $\MM_{(3)}$ corresponds to the graded comodule $A[\Delta^{-1}]$ seen as sitting in degree $1$. The reason is essentially that $\Phi_{r,s,t,u}^*\omega_0 = u \omega_0$, where $\omega_0 = \frac{dx}{2y+a_1x+a_3}$ is the usual invariant differential over the Weierstraß elliptic curve $E$ over $\Spec A[\Delta^{-1}]$ (i.e.\ a trivialization of $\omega$) and $\Phi_{r,s,t,u}$ is the endomorphism of $E$ classified by $r,s,t,u\in A$ (as in \cite{Rez02}, Section 9.2 and Proposition 9.4). 

Bauer computes in \cite{Bau08} the cohomology of the graded Hopf algebroid $(A, A[r,s,t])$, i.e.\ the Ext-groups $\Ext^l_{(A,A[r,s,t])\grcomod}(A, A[k])$ where $A[k]_m = A_{k+m}$ denotes an index shift of $A$. This is by the previous discussion, after inverting $\Delta$, isomorphic to 
\[\Ext^l_{\OO_{\MM_{(3)}}}(\OO, \omega^k) \cong H^l(\MM_{(3)}; \omega^k).\] \vspace{0.5cm}

Summarizing his calculation, we have
\[H^1(\MM_{(3)}; \omega^i) = \begin{cases}
  \Z/3\Z  & \text{if } i \equiv 2 \mod 12,\\
  0 & \text{else, }
\end{cases}\]
\[H^2(\MM_{(3)};\omega^i) = \begin{cases}
  \Z/3\Z  & \text{if } i \equiv 6 \mod 12,\\
  0 & \text{else. }
\end{cases}\]
Chosen generators of $H^1(\MM_{(3)};\omega^2)$ and $H^2(\MM_{(3)};\omega^6)$ are denoted by $\alpha$ and $\beta$. The algebra $H^*(\MM_{(3)};\omega^*)$ is for cohomological degree $>0$ generated over $\Z/3$ by $\alpha$, $\beta$ and $\Delta^{\pm 1}$ with only relation $\alpha^2 = 0$.

\subsection{Examples of Indecomposable Vector Bundles}\label{ExIndec}
The aim of this subsection is to define vector bundles $E_\alpha$ and $f_*f^*\OO$, compute their cohomology groups and deduce that these vector bundles are indecomposable. 

For brevity, we denote the structure sheaf $\OO_{\MM_{(3)}}$ by $\OO$ and all $\Ext$-groups will be in the category of $\OO$-modules. The class $\alpha \in H^1(\MM_{(3)};\omega^2) \cong \Ext^1(\omega^{-2}, \OO)$ classifies an extension
\begin{align}\label{mother}0 \to \OO \to E_\alpha \to \omega^{-2} \to 0.\end{align}
Up to isomorphism, which is not necessarily the identity on $\omega^j$, all non-trivial extensions of two line bundles are given as 
\[0 \to \omega^j \to E_\alpha \tensor \omega^j \to \omega^{j-2} \to 0\]
since the extension classified by $-\alpha$ is isomorphic to the one classified by $\alpha$ (via multiplication by $-1$).

We now want to compute some further Ext-groups.
\begin{prop}\label{CohEa}We have \[\Ext^1(\omega^j, E_\alpha) = \begin{cases}
  \Z/3\Z  & \text{if } j \equiv -4 \mod 12,\\
  0 & \text{else, }
\end{cases}\]
\[\Ext^2(\omega^j, E_\alpha) = \begin{cases}
  \Z/3\Z  & \text{if } j \equiv -6 \mod 12,\\
  0 & \text{else. }
\end{cases}\]

Furthermore, left multiplication with $\beta$ defines isomorphisms $\Ext^{i}(\omega^j, E_\alpha) \cong \Ext^{i+2}(\omega^j, E_\alpha)$. 

We denote that generator of the $\Ext^1$-group by $\at$ that maps to $\alpha$ under the map 
\[\Ext^1(\omega^{-4}, E_\alpha)\cong \Ext^1(\omega^{-2}, \omega^2\tensor E_\alpha) \to \Ext^1(\omega^{-2}, \OO)\]
induced by $E_\alpha\to\omega^{-2}$ as in (\ref{mother}). \end{prop}
\begin{proof}
 We have an exact sequence

\begin{center}\begin{tikzpicture}[>=angle 90]
\matrix[matrix of math nodes,row sep=3em, column sep=3em,
text height=1.5ex, text depth=0.25ex]
{&&|[name=0o2]|\Hom(\omega^j, \omega^{-2})\\
|[name=1o0]| \Ext^1(\omega^j, \OO)&|[name=1Ea]| \Ext^1(\omega^j, E_\alpha) &|[name=1o2]| \Ext^1(\omega^j, \omega^{-2})\\
|[name=2o0]| \Ext^2(\omega^j, \OO) &|[name=2Ea]| \Ext^2(\omega^j, E_\alpha) &|[name=2o2]| \Ext^2(\omega^j, \omega^{-2})\\
|[name=3o0]|\Ext^3(\omega^j, \OO) & |[name=dots]| \cdots &\\};
\draw[->,font=\scriptsize]
          (1o0) edge (1Ea)
          (1Ea) edge (1o2)
          (2o0) edge (2Ea)
          (2Ea) edge (2o2)
	  (3o0) edge (dots)
          (0o2) edge[out=-5,in=175] node[above left] {$\delta_0$} (1o0)
          (1o2) edge[out=-5,in=175] node[above left] {$\delta_1$} (2o0)
	  (2o2) edge[out=-5,in=175] node[above left] {$\delta_2$} (3o0);
\end{tikzpicture}\end{center}

To handle this, we need the following lemma:
\begin{lemma}[\cite{ML63}, II.9.1]\label{extmult}Let \[0\to A\to B\to C\to 0\] be an extension in an abelian category $\mathcal{A}$ (with enough injectives or projectives), corresponding to an Ext-class $x\in \Ext^1(C,A)$, and $T$ an arbitrary object in $\Ac$. The boundary map $\Ext^k(T, C) \to \Ext^{k+1}(T,A)$ of the long exact sequence for Ext-groups out of $T$ equals right multiplication by $x$. Similarly, the boundary map $\Ext^k(A, T) \to \Ext^{k+1}(C,T)$ of the sequence for Ext-groups into $T$ equals left multiplication by $x$. \end{lemma}

The map $\delta_0$ is therefore surjective, $\delta_{2k-1}$ is zero (since $(\beta^k\alpha\Delta^l)\cdot \alpha =0$) and $\delta_{2k}$ is an isomorphism for $k>0$. Hence, we get isomorphisms $\Ext^{2k-1}(\omega^j,E_\alpha) \cong \Ext^{2k-1}(\omega^j, \omega^{-2})$ and $\Ext^{2k}(\omega^j, E_\alpha) \cong \Ext^{2k}(\omega^j, \OO)$ for $k>0$. These isomorphisms commute with left multiplication by $\beta$ (as these isomorphisms are induced by post composing with a map of sheaves). This results in the Ext-groups given in the proposition. \end{proof}

\begin{cor}The vector bundle $E_\alpha$ is indecomposable.\end{cor}
\begin{proof}If $E_\alpha \cong \omega^k\oplus \omega^l$, then $\Ext^1(\omega^j, E_\alpha) \neq 0$ for $j \equiv k-2 \mod 12$ and $j \equiv l-2 \mod 12$. If $k \not\equiv l \mod 12$, this is a contradiction to the last proposition. If $k\equiv l \mod 12$, then $\Ext^1(\omega^{k-2}, E_\alpha) \cong (\Z/3)^2$, which is also a contradiction to the last proposition.\end{proof}

By dualizing the defining extension of $E_\alpha$ and tensoring with $\omega^{-2}$, we get an extension
\[0\to \OO \to \check{E_\alpha}\otimes\omega^{-2} \to \omega^{-2} \to 0,\] where $\check{E_\alpha}$ denotes the dual of $E_\alpha$ -- note that dualizing is here exact since $\omega^{-2}$ is a vector bundle. This extension is non-split (else the dual sequence would split as well) and therefore we have the following:
\begin{lemma}\label{Eadual}We have an isomorphism $\check{E_\alpha} \cong E_\alpha\tensor \omega^2$\end{lemma}
 Now consider the following lemma:

\begin{lemma}\label{ExtAdjunction}Let $(\XX, \OO)$ be a ringed site, $\EE$ and $\FF$ be vector bundles and $\GG$ be a quasi-coherent sheaf on $\XX$. Then we have $\Ext^i(\EE, \FF\tensor \GG) \cong \Ext^i(\EE\tensor \check{\FF}, \GG)$.\end{lemma}

\begin{proof}Since vector bundles are (strongly) dualizable, we have a natural isomorphism \[\mathcal{H}om_{\OO}(\EE, \FF\tensor \GG) \cong \mathcal{H}om_{\OO}(\EE \tensor \check{\FF},\GG)\] of Hom-sheaves. The same holds for all higher Ext-sheaves (they are all zero).
Therefore, \[\Ext^i(\EE,\FF\tensor\GG)\cong H^i(\XX; \mathcal{H}om_{\OO}(\EE,\FF\tensor\GG))\cong H^i(\XX;\mathcal{H}om_{\OO}(\EE \tensor \check{\FF},\GG)) \cong \Ext^i(\EE\tensor\check{\FF},\GG)\] by the Grothendieck spectral sequence converging from the cohomology of the Ext-sheaves to the Ext-groups.
\end{proof}

In particular, we have \[\Ext^i(E_\alpha\tensor\omega^j,\OO) \cong \Ext^i(\omega^j, \check{E_\alpha}) \cong \Ext^i(\omega^{j-2}, E_\alpha).\] This vanishes for $i=1$ iff  $j\not\equiv -2 \mod 12$. 

\begin{remark}One possibility to construct $E_\alpha$ (or rather $\check{E_\alpha})$ concretely is the following: The reason for $E_\alpha$ not to split is the same as for the Hasse invariant not to lift, namely the non-vanishing of $H^1(\MM_{(3)}; \omega^2)$. At the prime $3$, the Hasse invariant is a section $A\in H^0(\MM_{(3)}; \omega^2/3)$, i.e.\ a morphism $\OO \to \omega^2/3$. Define a vector bundle $E$ as the fiber product $\OO \times_{\omega^2/3} \omega^2$. Since $\omega^2\to \omega^2/3$ is surjective, $E \to \OO$ is surjective and non-split (since a splitting corresponded to a lifting of the Hasse invariant). The kernel is given by the map $\omega^2 \to E$ inducing $\omega^2\xrightarrow{3} \omega^2$ and $\omega^2 \xrightarrow{0} \OO$. Thus, we get a non-split extension
\[0\to\omega^2 \to E \to \OO \to 0.\]
Thus, $E \cong E_\alpha \tensor \omega^2 \cong \check{E_\alpha}$ by Lemma \ref{Eadual}.\end{remark}

\vspace{0.5cm}

A further, particularly important example of a vector bundle is the following: Let \[f\co \MM_0(2)_{(3)} \to \MM_{(3)}\] be the usual projection map from the moduli stack of elliptic curves with chosen point of exact order $2$. Then $f_*f^*\omega^j = f_*f^*\OO \tensor \omega^j$ defines a family of rank $3$ vector bundles on $\MM_{(3)}$. 

\begin{lemma}\label{VanishingofCohomology}The cohomology groups $H^i(\MM_{(3)}; f_*\FF)$ vanish for $i>0$ for every quasi-coherent sheaf $\FF$ on $\MM_0(2)$.\end{lemma}
\begin{proof}The map $f$ is finite and, in particular, affine. Therefore, all higher direct images $R^if_*$ vanish and, using a degenerate form of the Leray spectral sequence, we get
\[H^i(\MM_{(3)}; f_*\FF) \cong H^i(\MM_0(2)_{(3)}; \FF).\]
 The latter vanishes by Lemma \ref{gradetriv} since $\MM_0(2)_{(3)} \simeq \Spec \Z_{(3)}[b_2,b_4,\Delta^{-1}]//\G_m$ by \cite[Section 1.3.2]{Beh06}. \end{proof}
\begin{cor}The vector bundle $f_*f^*\OO$ is indecomposable.\end{cor}
\begin{proof}If it is not indecomposable, it is a sum of a line $\LL$ and a plane bundle $\EE$. By the classification of line bundles and Section \ref{EllCoh}, there exists some $j\in\Z$ such that $\LL \tensor \omega^j$ has non-trivial first cohomology, but $f_*f^*\OO \tensor \omega^j \cong f_*f^*\omega^j$ has trivial first cohomology.\end{proof} 

Note that the last lemma implies $\Ext^i(\EE, f_*f^*\OO) = 0$ for all $i>0$ for every vector bundle $\EE$. Indeed, we have a (Grothendieck) spectral sequence
\[H^s(\mathcal{E}xt^t_{\OO}(\EE, f_*f^*\OO)) \Rightarrow \Ext^{s+t}(\EE, f_*f^*\OO).\]
Since $\EE$ is a vector bundle, all Ext-sheaves for $t>0$ vanish. Furthermore, $\mathcal{H}om(\EE, f_*f^*\OO) \cong f_*\mathcal{H}om_{\OO_{\MM_0(2)_{(3)}}}(f^*\EE, \OO_{\MM_0(2)_{(3)}})$. Thus its cohomology groups vanish for $s>0$ by the last lemma and the spectral sequence is concentrated in $s= t= 0$. Likewise $\Ext^i(f_*f^*\OO, \EE) = 0$ for all $i>0$ by Lemma \ref{ExtAdjunction} since $f_*f^*\OO$ is self-dual by Lemma \ref{RepLem2} in the next subsection. To show this lemma (and more), it will be convenient to do a short excursion to representation-theoretic methods. But first, let us summarize the preceeding discussion:
\begin{prop}\label{ExtEaa}Every extension of vector bundles 
 \[ 0\to \EE \to \FF \to \GG \to 0\]
on $\MM_{(3)}$ splits if $\EE$ or $\GG$ is isomorphic to $f_*f^*\OO\tensor \omega^j$ for some $j$. 
\end{prop}

\subsection{Representation-Theoretic Methods}
In this subsection, we will use integral representations of $S_3$ to produce vector bundles on $\MM_{(3)}$. We will recover by these means the examples of the last subsection, but in a more concrete setting. This will allow us to prove that $f_*f^*\OO$ is self-dual and to produce the short exact sequences (\ref{firstoftwo}) and (\ref{secondoftwo}). 

 We denote by $p\co \MM(2)_{(3)} \to \MM_{(3)}$ the forgetful morphism from the moduli stack of elliptic curves with level-$2$-structure at the prime $3$. This $S_3$-torsor is classified by a map $i\co \MM_{(3)} \to \Spec \Z_{(3)}//S_3$, which fits into a (2-)commutative square 
 \[\xymatrix{\MM(2)_{(3)} \ar[r] \ar[d] & \Spec \Z_{(3)} \ar[d]\\
  \MM_{(3)} \ar[r]^-i & \Spec \Z_{(3)}//S_3}
 \]

 Set $\Lambda_n = \omega^n(\MM(2)_{(3)})$. By \cite[Proposition 7.1]{Sto12}, we have $\MM(2)_{(3)} \simeq \Spec \Lambda_*//\G_m$ and thus $\QCoh(\MM(2)_{(3)}) \simeq \Lambda_*\grmod$. This induces a diagram
 \[\xymatrix{S_3-\Lambda_*\grmod   & S_3-\Z_{(3)}\modules\ar[l]\\
  \QCoh(\MM_{(3)}) \ar[u] & \QCoh(\Spec \Z_{(3)}//S_3)\ar[l]^{i^*}\ar[u]}
 \]
 The upper left corner consists of graded $\Lambda_*$-modules with semilinear $S_3$-action (with respect to the $S_3$-action on $\Lambda_*$ induced by that on $\MM_{(3)}(2)$). The two vertical maps are equivalences by Galois descent. We denote the composition of the inverse of the right vertical equivalence with $i^*$ by $I\co \Z_{(3)}[S_3]\modules \to \QCoh(\MM_{(3)})$. Viewing the right hand side as graded $S_3$-equivariant $\Lambda_*$-modules, this corresponds just to the functor $M\mapsto M\tensor \Lambda_*$ with diagonal $S_3$-action (i.e. the upper horizontal functor). Thus, we have $S_3$-equivariantly 
\[I\Z_{(3)}[S_3](\Spec \Lambda_*) \cong \bigoplus_{S_3}\Lambda_*;\]
 here we let $S_3$ act on $S_3$ from the left by $h \cdot g = gh^{-1}$; on the right hand side $S_3$ acts simultaneously by permuting the factors (by the action just described) and on $\Lambda_*$. This convention is chosen for the following reason: Consider the map 
  \[\MM_{(3)}(2)\times S_3 \xrightarrow{\simeq} \MM_{(3)}(2) \times_{\MM_{(3)}}\MM_{(3)}(2)\]
 indicated by the formula $(m,g) \mapsto (m, gm)$. If $S_3$ acts just on the left factor in the right hand side, the map becomes equivariant if we act on $S_3\times\MM_{(3)}(2)$ via $h\cdot(g,m) = (gh^{-1}, hm)$. We can base change this by $\Spec \Lambda_* \times_{\MM_{(3)}(2)}$ to get an equivalence 
\[\Spec\Lambda_*\times S_3 \xrightarrow{\simeq} \Spec\Lambda_* \times_{\MM_{(3)}}\MM_{(3)}(2)
\]
Thus, we have also $p_*p^*\OO(\Spec \Lambda_*) \cong_{S_3} \bigoplus_{S_3}\Lambda_*$ and therefore $I\Z_{(3)}[S_3] \cong p_*p^*\OO$ for \[p\co \MM(2)_{(3)}\to \MM_{(3)}\] the forgetful map as above. Note that it is true in general that $IM$ is a vector bundle of rank $n$ if $M\in\Z_{(3)}[S_3]\modules$ is free of rank $n$ as a $\Z_{(3)}$-module.

 Similarly, for $P$ the standard permutation representation of $S_3$ of rank $3$, we have that $IP \cong f_*f^*\OO$ (since $\MM_{(3)}(2)\times_{\MM_{(3)}} \MM_0(2)_{(3)} \simeq_{S_3} \coprod_{\{1,2,3\}}\MM_{(3)}(2)$). As 
\[\coprod_{\{1,2,3\}}\MM_{(3)}(2) \to \MM_{(3)}(2)\times_{\MM_{(3)}}\MM_0(2)_{(3)} \to \MM_{(3)}(2)\times_{\MM_{(3)}}\MM_{(3)}\simeq \MM-{(3)}(2)\] 
is the fold map $\coprod_{\{1,2,3\}}\MM_{(3)}(2)\to \MM_{(3)}(2)$, the functor $I$ sends the diagonal map $\Z_{(3)}\to P$ to the adjunction unit $\OO\to f_*f^*\OO$. \vspace{0.5cm}

 The group $S_3$ acts on $\Z_{(3)}[\zeta_3]$ for $\zeta_3$ a third root of unity via permuting the ordered set $(1, \zeta_3, \zeta_3^2)$ of roots of unity. Denote the basis vector corresponding to $1,2,3$ in $P$ by $t_1,t_2, t_3\in P$. We have two exact sequences
 \begin{align}\label{RepSeq1}0 \to \Z_{(3)} \to P \to \Z_{(3)}[\zeta_3] \to 0 \end{align}
 and 
 \begin{align}\label{RepSeq2}0 \to (1-\zeta_3)\Z_{(3)}[\zeta_3] \to P \to \Z_{(3)} \to 0\end{align}
 of $\Z_{(3)}[S_3]$-modules (sending $t_1$ to $1$ respectively $1-\zeta_3$ to $t_1-t_2$). Here, the map $\Z_{(3)} \to P$ is the diagonal and the map $P \to \Z_{(3)}$ is the summing map. Since $i$ is flat, $I$ is exact and we get exact sequences
 \begin{align}\label{VecSeq1}0 \to \OO \to f_*f^*\OO \to I\Z_{(3)}[\zeta_3] \to 0 \end{align}
 and
 \begin{align}\label{VecSeq2}0 \to I\left((1-\zeta_3)\Z_{(3)}[\zeta_3]\right) \to f_*f^*\OO \to \OO \to 0. \end{align}

\begin{prop}We have $\Lambda_* \cong \Z_{(3)}[\lambda_1,\lambda_2,\Delta^{-1}]$ with $\Delta = 16\lambda_1^2\lambda_2^2(\lambda_1-\lambda_2)^2$ and $|\lambda_1| = |\lambda_2| = 2$. Furthermore, the sub-$S_3$-representation $\Z_{(3)} \langle \lambda_1,\lambda_2\rangle\subset \Lambda_2$ (where $\Z_{(3)}\langle \lambda_1,\lambda_2\rangle$ denotes the free $\Z_{(3)}$-module of rank $2$ on generators $\lambda_1$ and $\lambda_2$) is isomorphic to $(1-\zeta_3)\Z_{(3)}[\zeta_3]$.\end{prop}
\begin{proof}This follows from the existence of the Legendre normal form, but we will give precise references. 

The first statement is contained in the discussion before \cite[Proposition 7.1]{Sto12} except for the formula for $\Delta$, which follows by relating the $\lambda_i$ to the usual $b_i$. 

A formula for the action of $S_3$ on $\Lambda_*$ is given in \cite[Lemma 7.3]{Sto12}. It follows that the map $\Z_{(3)} \langle \lambda_1,\lambda_2\rangle \to (1-\zeta_3)\Z_{(3)}[\zeta_3]$ given by $\lambda_1 \mapsto \zeta_3-1$ and $\lambda_2\mapsto \zeta_3^2-1$ is $S_3$-equivariant and thus a $S_3$-equivariant isomorphism.\end{proof}


\begin{lemma}\label{RepLem1}We have $I\left((1-\zeta_3)\Z_{(3)}[\zeta_3]\right) \cong \omega^4\tensor E_\alpha$, with notation as in the last subsection.\end{lemma}
\begin{proof}
Since the higher cohomology of $f_*f^*\OO$ vanishes, we have by the long exact sequence associated to extension (\ref{VecSeq2}) that 
\[H^2_k(\MM_{(3)}; I\left((1-\zeta_3)\Z_{(3)}[\zeta_3]\right)) = \begin{cases}\Z/3 & \text{ for } k\equiv 2 \mod 12 \\
0 & \text{ else.} \end{cases} \]
Thus, the rank $2$ vector bundle $I\left((1-\zeta_3)\Z_{(3)}[\zeta_3]\right)$ is indecomposable (since every line bundle has non-trivial second cohomology in every $12$-th degree).

By the last proposition, the multiplication map $\Lambda_*\tensor \Lambda_* \to \Lambda_*$ restricts to an $S_3$-equivariant map $\Lambda_* \tensor (1-\zeta_3)\Z_{(3)}[\zeta_3] \to \Lambda_{*+2}$, which is surjective (since $\lambda_2$ is a unit as it is a divisor of $\Delta$). By Galois descent, this induces in turn a surjective map $I\left((1-\zeta_3)\Z_{(3)}[\zeta_3]\right) \to \omega^2$. Its kernel is a line bundle (since locally every surjection onto a vector bundle splits). The vector bundle $I\left((1-\zeta_3)\Z_{(3)}[\zeta_3]\right)$ is thus a non-split extension of $\omega^2$ with another line bundle. Hence, by the considerations at the beginning of Subsection \ref{ExIndec}, it has to be isomorphic to $E_\alpha\tensor \omega^4$.\end{proof}

\begin{lemma}\label{RepLem2}We have $I\Z_{(3)}[\zeta_3] \cong \omega^{-2} \tensor E_\alpha$ and $f_*f^*\OO$ is self-dual.\end{lemma}
\begin{proof} Equip $\check{P} = \Hom(P, \Z_{(3)})$ with the $S_3$-action $(g\cdot f)(p) = f(g^{-1}(p))$. Then sending each basis vector of $P$ to its dual vector defines an $S_3$-equivariant isomorphism $P\cong \check{P}$. With this identification, the dual of the diagonal is the summing map. Thus, the short exact sequences (\ref{RepSeq1}) and (\ref{RepSeq2}) are dual to each other and hence the dual of $(\Z_{(3)}[\zeta_3])$ is isomorphic to $(1-\zeta_3)\Z_{(3)}[\zeta_3]$. Since $i$ is a flat map, $I = \text{equivalence}\, \circ\, i^*$ sends duals to duals (see \cite[IV, Proposition 18]{Ser00} for the commutative algebra statement). Hence, $I\Z_{(3)}[\zeta_3] \cong \omega^{-2} \tensor E_\alpha$ by Lemma \ref{Eadual} and $(f_*f^*\OO)\check{} \cong f_*f^*\OO$.\end{proof}

Thus, the short exact sequences (\ref{VecSeq1}) and (\ref{VecSeq2}) are isomorphic to
\begin{eqnarray}\label{firstoftwo}0 \to \OO \to f_*f^*\OO \to E_\alpha\otimes \omega^{-2}\to 0 \end{eqnarray}
and
\begin{eqnarray}\label{secondoftwo}0 \to E_\alpha \otimes \omega^4 \to f_*f^*\OO \to \OO \to 0 \end{eqnarray}
where the map $\OO \to f_*f^*\OO$ is the adjunction map and the map $f_*f^*\OO \to \OO$ is its dual (under a chosen isomorphism $(f_*f^*\OO)\check{} \cong f_*f^*\OO$). 

The two extensions (\ref{firstoftwo}) and (\ref{secondoftwo}) are non-split (as can be seen by computing cohomology). Thus, extension (\ref{secondoftwo}) corresponds to $\pm$ the $\Ext^1$-class $\widetilde{\alpha}$ introduced in Proposition \ref{CohEa}. 

\subsection{Extensions of $E_\alpha$ and $E_\alpha \tensor \omega^j$}
\begin{prop}\label{EaExtension}Let 
 \[0 \to E_\alpha \to Y \to E_\alpha\tensor \omega^j \to 0\]
be a non-split extension. Then $j\equiv -2 \mod 12$ and $Y \cong E_\alpha \tensor E_\alpha \cong f_*f^*\OO\oplus \omega^{-2}$.
\end{prop}
\begin{proof}
If we tensor the extension (\ref{firstoftwo}) with $E_\alpha$, we get:
\[0 \to E_\alpha\to f_*f^*\OO\tensor E_\alpha \to E_\alpha\tensor E_\alpha \tensor \omega^{-2}\to 0\]
The middle term is isomorphic to $f_*f^*\OO\oplus f_*f^*\OO\tensor \omega^{-2}$ (as can be seen by tensoring the extension (\ref{mother}) with $f_*f^*\OO$) and therefore $\Ext^i(f_*f^*\OO\tensor E_\alpha, \FF) = 0$ for every vector bundle $\FF$ and $i>0$ by Proposition \ref{ExtEaa}. Thus, using Lemma \ref{ExtAdjunction} for the second isomorphism, 
\[\Ext^2(\omega^{j-4}, E_\alpha)\cong \Ext^1(\omega^{j-2}, E_\alpha\tensor E_\alpha) \cong  \Ext^1(E_\alpha\tensor\omega^j, E_\alpha),\] 
which is zero unless $j\equiv -2 \mod 12$, when it is isomorphic to $\mathbb{Z}/3$, by Proposition \ref{CohEa}. Tensoring (\ref{mother}) with $E_\alpha$ gives an extension
 \[0\to E_\alpha \to E_\alpha\tensor E_\alpha \to E_\alpha\tensor\omega^{-2}\to 0,\] 
which is non-split since $H^1(\MM_{(3)}; E_\alpha\tensor E_\alpha\tensor \omega^6) = 0 $ differs from $H^1(\MM_{(3)}; (E_\alpha\oplus E_\alpha\tensor\omega^{-2})\tensor \omega^6) \cong \Z/3$ by the calculation above and Proposition \ref{CohEa}. It follows that this extension represents a generator of $\Ext^1(E_\alpha\tensor \omega^{-2}, E_\alpha)$. 

Consider now the extension $X$ corresponding to an element in $\Ext^1(E_\alpha\tensor\omega^{-2},E_\alpha)$ coming from a generator in $\Ext^1(E_\alpha\tensor\omega^{-2},\OO)$ (corresponding to (\ref{firstoftwo})) via the map induced by $\OO\to E_\alpha$. This extension sits in a diagram

\[\xymatrix{&0\ar[d]&0\ar[d]&&\\
0\ar[r]&\OO\ar[r]\ar[d]&f_*f^*\OO\ar[r]\ar[d]&E_\alpha\tensor \omega^{-2}\ar[r]\ar[d]^=&0\\
0\ar[r]&E_\alpha\ar[r]\ar[d]&X \ar[d]\ar[r]&E_\alpha\tensor\omega^{-2}\ar[r]&0\\
&\omega^{-2}\ar[d]\ar[r]^\cong&\omega^{-2} \ar[d]&&\\
&0&0&&}\]

\noindent with rows and columns exact. This implies $X \cong f_*f^*\OO\oplus \omega^{-2}$ by Proposition \ref{ExtEaa}. Thus, the cohomology of $X\tensor \omega^6$ differs from that of $(E_\alpha\oplus E_\alpha\tensor\omega^{-2})\tensor \omega^6$ and the middle horizontal extension is non-split. Hence, $E_\alpha\tensor E_\alpha \cong X$ as $\Ext^1(E_\alpha\tensor \omega^{-2}, E_\alpha) = \Z/3$.. 
\end{proof}

\section{Classification of Iterated Extensions of Line Bundles on $\MM_{(3)}$}\label{Classifying}
In this section, we want to classify all iterated extensions of line bundles on $\MM_{(3)}$. For shortness, we will call iterated extensions of line bundles \textit{standard vector bundles}. More precisely:
\begin{defi}\label{AlgebraStand}We define the notion of a \textit{standard vector bundle} on a ringed site $\XX$ inductively on the rank as follows: Every line bundle on $\XX$ is called \textit{standard}. Furthermore, a vector bundle $\EE$ on $\XX$ is called \textit{standard} if there is an injection $\LL \hookrightarrow \EE$ from a line bundle on $\XX$ such that the cokernel is a standard vector bundle. \end{defi}
Equivalently, we can characterize the class of standard vector bundles as the smallest sub-class of all vector bundles which is closed under extensions and contains all line bundles. 

\begin{thm}[Theorem B]Every standard vector bundle on $\MM_{(3)}$ is a direct sum of the form $\bigoplus_I \omega^{n_i} \oplus \bigoplus_J E_\alpha\tensor\omega^{n_j} \oplus \bigoplus_K f_*f^*\OO\tensor \omega^{n_k}$.\end{thm}
\begin{proof}We will prove this theorem by induction on the rank of the vector bundle. The rank $1$ case is the classification of line bundles.

So assume that we have proven the theorem for all standard vector bundles of rank smaller than $n$ and that $X$ is a standard vector bundle of rank $N$. By the induction hypothesis, we have a short exact sequence
\begin{eqnarray}\label{vst}0\to \omega^q \to X \to Y \to 0\end{eqnarray}

\noindent where $Y$ is of the form \[\bigoplus_{I_Y} \omega^{n_i} \oplus \bigoplus_{J_Y} E_\alpha\tensor\omega^{n_j} \oplus \bigoplus_{K_Y} f_*f^*\OO\tensor \omega^{n_k}\] and of rank $(N-1)$. We will call the depicted summands of $Y$ its \textit{standard summands}. Furthermore, we can assume that there is no vector bundle $Z$ which is isomorphic to \[\bigoplus_{I_Z} \omega^{m_i} \oplus \bigoplus_{J_Z} E_\alpha\tensor\omega^{m_j} \oplus \bigoplus_{K_Z} f_*f^*\OO\tensor \omega^{m_k}\] with $|I_Z| < |I_Y|$ such that there is a surjective morphism $X \to Z$ with a line bundle as kernel. In addition, we assume (for notational simplicity) that $q=0$.

We assume for contradiction that $X$ is not of the form which is demanded by the theorem we want to prove. Then the extension (\ref{vst}) is non-trivial. Since the Ext functor commutes with (finite) direct sums, there is at least one standard summand $S$ of $Y$ such that the map $\Ext^1(Y, \OO) \to \Ext^1(S,\OO)$ (induced by the inclusion) sends the class $x\in\Ext^1(Y,\OO)$ corresponding to (\ref{vst}) to a non-trivial class. We will prove the theorem case by case:

1) $S = f_*f^*\OO\tensor \omega^j$: this cannot happen since $\Ext^1(f_*f^*\OO\tensor \omega^j, \OO) = 0$ by Proposition \ref{ExtEaa}. 

2) $S = E_\alpha\tensor \omega^j$: Since the only non-split extension of $\OO$ and an $E_\alpha\tensor \omega^j$ is $f_*f^*\OO$ with $j=-2$ (as follows from Proposition \ref{CohEa}), we get a diagram (with rows and columns exact) of the form:

\[\xymatrix{&&0\ar[d]&0 \ar[d]&\\
0\ar[r]&\OO\ar[r]\ar[d]^=&f_*f^*\OO\ar[r]\ar[d]&E_\alpha\tensor \omega^{-2}\ar[r]\ar[d]&0\\
0\ar[r]&\OO\ar[r]&X \ar[d]\ar[r]&Y\ar[r]\ar[d]&0\\
&&Y-(E_\alpha\tensor \omega^{-2}) \ar[d]\ar[r]^=& Y - (E_\alpha\tensor \omega^{-2})\ar[d]&\\
&&0&0&}\]

The left vertical extension is trivial since $\Ext^1(Y-(E_\alpha\tensor \omega^{-2}), f_*f^*\OO) = 0$ by Proposition \ref{ExtEaa}. Therefore, 
\[X \cong f_*f^*\OO\oplus (Y-(E_\alpha\tensor \omega^{-2})).\] 

3) $S = \omega^j$: Since the only non-split extension of $\OO$ and an $\omega^j$ is $E_\alpha$ with $j=-2$, we get a diagram (with rows and columns exact) of the form:

\[\xymatrix{&&0\ar[d]&0\ar[d]&\\
0\ar[r]&\OO\ar[r]\ar[d]^=&E_\alpha\ar[r]\ar[d]&\omega^{-2}\ar[r]\ar[d]&0\\
0\ar[r]&\OO\ar[r]&X \ar[d]\ar[r]&Y\ar[r]\ar[d]&0\\
&&Y-\omega^{-2}\ar[d]\ar[r]^=&Y - \omega^{-2} \ar[d]&\\
&&0&0&}\]

If the left vertical extension in the diagram is non-split, there is a standard summand $S'$ of $Y-\omega^{-2}$ such that the map $\Ext^1(Y-\omega^{-2}, E_\alpha) \to \Ext^1(S',E_\alpha)$ induced by the inclusion sends the element $y\in\Ext^1(Y-\omega^{-2},E_\alpha)$ corresponding to the left vertical extension to a non-trivial class. If $S' \cong \omega^l$, then the argument is similar to the case before and we get $X \cong (Y-\omega^{-2} - \omega^{-4})\oplus f_*f^*\OO$. The case $S' \cong f_*f^*\OO\tensor\omega^l$ can again not occur because $\Ext^1(f_*f^*\OO\tensor\omega^l, E_\alpha) = 0$. Therefore, we can assume that $S'$ is isomorphic to $E_\alpha\tensor \omega^l$ for some $l$. By Proposition \ref{EaExtension}, the only non-trivial extension of $E_\alpha$ with some $E_\alpha \tensor \omega^l$ is $f_*f^*\OO\oplus\omega^{-2}$ with $l=-2$. So we can assume that we get a commutative diagram (with rows and columns exact) of the form:
\[\xymatrix{&&0\ar[d]&0\ar[d]&\\
   0\ar[r]&E_\alpha\ar[r]\ar[d]^=&f_*f^*\OO\oplus\omega^{-2}\ar[r]\ar[d]&E_\alpha\tensor\omega^{-2}\ar[r]\ar[d]&0\\
   0\ar[r]&E_\alpha\ar[r]&X\ar[d]\ar[r]&Y-\omega^{-2}\ar[r]\ar[d]&0\\
   &&Y-\omega^{-2}-(E_\alpha\tensor\omega^{-2})\ar[d]\ar[r]^=&Y-\omega^{-2}-(E_\alpha\otimes\omega^{-2})\ar[d]&\\
   &&0&0&
}\]

Pushing the left vertical extension forward along the projection map $f_*f^*\OO \oplus \omega^{-2} \to f_*f^*\OO$ produces the following diagram (with rows and columns exact): 
\[\xymatrix{&0\ar[d]&0\ar[d]&&\\
&\omega^{-2} \ar[r]^= \ar[d] & \omega^{-2} \ar[d]& &\\
0 \ar[r] & f_*f^*\OO \oplus \omega^{-2} \ar[r] \ar[d] & X \ar[r] \ar[d] & Y-\omega^{-2}-(E_\alpha\otimes\omega^{-2}) \ar[r] \ar[d]^=& 0 \\
0 \ar[r] & f_*f^*\OO										\ar[r] \ar[d] & Z \ar[r] \ar[d] & Y-\omega^{-2}-(E_\alpha\otimes\omega^{-2}) \ar[r] & 0\\
&0&0&& } \]

The lower horizontal extension splits by Proposition \ref{ExtEaa} so that 
\[Z \cong f_*f^*\OO \oplus (Y - \omega^{-2}-(E_\alpha\otimes\omega^{-2})).\] Thus, there is a surjective map $X \to Z$ with a line bundle as kernel and $|I_Z| = |I_Y| - 1$ (where $I_Z$ is the set of standard line bundle summands as above), which is a contradiction to the minimality of $|I_Y|$. 

This completes the proof of Theorem A.
\end{proof}

\begin{cor}Let $\EE$ be a standard vector bundle on $\MM_{(3)}$ with $H^1(\MM_{(3)}; \EE\tensor \omega^k) = 0$ for all $k\in\Z$. Then the rank of $\EE$ is divisible by $3$.\end{cor}
\begin{proof}By the main theorem of this section, $\EE$ is a sum of vector bundles of the form $\omega^j$, $E_\alpha \tensor \omega^j$ and $f_*f^*\OO\tensor \omega^j$. But $\omega^j\tensor \omega^{2-j}$ and $E_\alpha\tensor \omega^j\tensor \omega^{4-j}$ have non-trivial first cohomology. Therefore, $\EE$ is a sum of vector bundles of the form $f_*f^*\OO\tensor \omega^j$. These have rank $3$. \end{proof}

\section{Vector Bundles on $\MM_{(2)}$}\label{vec2}
We will turn now our attention to $\MM_{(2)} = \MM_{\Z_{(2)}}$, where the situation is in some respects quite different from bigger primes; we will see that we have here infinitely many indecomposable vector bundles (of arbitrary high rank). 

Recall that we have a $GL_2(\F_3)$-Galois covering $\MM(3)_{(2)} \to \MM_{(2)}$ for $\MM(3)$ the moduli stack of elliptic curves with level-$3$ structure at the prime $2$. Set $G= GL_2(\F_3)$. 

Consider the elliptic curve $E\co y^2 +y = x^3$ over $\overline{\F}_2$. This has, according to \cite{Sil09}, III.10.1, automorphism group $S$ of order $24$.  By \cite[2.7.2]{K-M85}, the morphism $S\to G$ (given by the operation of $S$ on $E[3]$) is injective. Using elementary group theory, we get that $G$ has a unique subgroup of order $24$, namely $SL_2(\F_3)$; thus $S$ embeds onto $SL_2(\F_3)$. The group $SL_2(\F_3)$ has as a $2$-Sylow group the quaternion group $Q$, the multiplicative subgroup of the quaternions generated by $i$ and $j$. This defines an action of $Q$ on $E$. 

Since the finite group scheme $E[3]$ over $\overline{\F}_2$ is isomorphic to $(\Z/3)^2$, we can choose a level-$3$ structure on $E$. This gives the following $2$-commutative diagram 
\[\xymatrix{\Spec \overline{\F}_2 \times Q \ar[r]\ar[d]^{\pr_1}& \MM(3)_{(2)}\times G \ar[r]\ar[d] & \Spec \Z \times G \ar[d]^{\pr_1}\\
\Spec{\overline{\F}_2} \ar[d]\ar[r] & \MM(3)_{(2)} \ar[d]\ar[r] & \Spec \Z \ar[d]\\
\Spec{\overline{\F}_2}//Q \ar[r]^e & \MM_{(2)} \ar[r]^{\pi} & \Spec \Z // G } \]

Thus, we get functors $I\co \Z[G]\modules \to \QCoh(\MM_{(2)})$ and $R\co \QCoh(\MM_{(2)}) \to \overline{\F}_2[Q]\modules$ given by $\pi^*$ and $e^*$, composed with the Galois descent equivalence, respectively. As $Q$ acts on $E[3]$ via the inclusion $Q\subset G$, the functor $RI$ is given upto to isomorphism by tensoring with $\overline{\F}_2$ and restricting to $Q\subset G$.

Our next aim is to define certain integral $G$-representations, which will allow us (via the functor $I$) to prove the existence of an infinite family of indecomposable vector bundles on $\MM_{(2)}$. There is a family of $C_2\times C_2$-representations over $\Z$ given as $M_n = \Z x_1\oplus  \cdots \oplus \Z x_n \oplus \Z y_0\oplus \cdots \oplus \Z y_n$ and 
\begin{eqnarray*}(g_1 + (-1)^i)x_i = y_{i-1}, \hspace{1cm}(g_2+(-1)^i)x_i = y_i,\\
 (g_1-(-1)^i)y_i = (g_2+(-1)^i)y_i = 0,
\end{eqnarray*}
where $g_1$ and $g_2$ generate $C_2\times C_2$ (see \cite{H-R62}, 6.2). The modules $\overline{M_n} = M_n \otimes_{\Z} \overline{\F}_2$ (and, hence, also the $M_n$) are indecomposable as $C_2\times C_2$-modules (see \cite{H-R61}, Proposition 5(ii) and its corollary). The same holds if we pull them back to representations of $Q$ via the surjective morphism $\rho\co Q\to C_2\times C_2$ (given by dividing out $i^2\in Q$). Indeed, as there is no $C_2\times C_2$-equivariant idempotent on $\overline{M_n}$, there is also no $Q$-equivariant idempotent on $\overline{M_n}$. By abuse of notation, we denote $\rho^*\overline{M_n}$ also by $\overline{M_n}$. 

Let $(Y_j)_{j\in J}$ be the collection of indecomposable vector bundles on $\MM_{(3)}$. Decompose $I\ind_Q^G M_n$ as $\bigoplus_{j\in J} a_j Y_j$ (with almost all $a_j = 0$). Thus, $RI\ind_Q^G M_n \cong \bigoplus_{j\in J}  a_j R(Y_j)$. Since 
\[RI\ind_Q^G M_n \cong \res_Q^G\ind_Q^G \overline{M_n} \cong \bigoplus_{G/Q}\overline{M_n},                                                                                                                                                                                    \]
we see that $\overline{M_n}$ is a direct summand of this module. Recall now the theorem of Krull--Remak--Schmidt:
\begin{thm}[Krull--Remak--Schmidt]Every noetherian and artinian module over a ring has a (up to permutation and isomorphisms) unique decomposition into indecomposable modules.\end{thm}
Thus, as $\overline{M_n}$ is finite-dimension as an $\overline{\F}_2$-vector space and hence noetherian and artinian, $\overline{M_n}$ has to be a summand of one of the $RY_j$. Since $\rk \overline{M_n} = 2n+1$, the rank of $RY_j$ (and hence of $Y_i$) must be at least $2n+1$. Therefore, $\MM_{(2)}$ has indecomposable vector bundles of arbitrary high rank. 

\appendix
\section{Quasi-Coherent Sheaves and Completions}
In this appendix we review some descent results for (quasi-)coherent sheaves on Artin stacks with respect to completions. Similar results appear for affine schemes and algebraic spaces instead of Artin stacks in \cite{F-R70} and \cite{MB96}. We will freely restrict to noetherian and separated situations when convenient. 

Throughout this section let $\XX$ be an Artin stack. In \cite{LMB00} and \cite[Definition 3.1]{Ols07}, the lisse-\'etale site corresponding to $\XX$ is defined: It is the full subcategory of $\XX$-schemes consisting of smooth $\XX$-schemes and a family $\{f_i\co U_i\to U\}$ is called a covering if each $f_i$ is \'etale and the maps are jointly surjective. Denote by $\XX_{lis-et}$ the corresponding topos. 

Call a sheaf of rings $\Ac$ on $\XX_{lis-et}$ \textit{flat} if for every smooth map $f\co U \to V$ of smooth $\XX$-schemes the associated map $f^{-1}\Ac_V \to \Ac_U$ is faithfully flat (\cite[Definition 3.7]{Ols07}). An example is the structure sheaf $\OO_\XX$. A sheaf of $\Ac$-modules is \textit{cartesian} if for every map $U\to V$ of smooth $\XX$-schemes the natural map
\[\Ac_U\tensor_{f^{-1}\Ac_V} f^{-1}\MM_V \to \MM_U\]
is an isomorphism. For $\Ac = \OO_\XX$ an $\OO_\XX$-module is cartesian if and only if it is quasi-coherent. Therefore, we denote the category of cartesian $\Ac$-modules by $\QCoh(\XX, \Ac)$ and the category of those cartesian $\Ac$-modules locally of finite presentation by $\Coh(\XX,\Ac)$.  

\begin{thm}[\cite{Ols07}, Proposition 4.4]\label{OlssonDes}Let $\XX$ be an Artin stack, $\Ac$ a flat sheaf of rings on $\XX_{lis-et}$ and $X \to \XX$ a smooth cover by an algebraic space. Denote by $X_\bullet$ the simplicial algebraic space given by $X_n \simeq X^{\times_\XX n+1}$. Denote furthermore by $\Mod_{\cart}(X^+_\bullet)$ the category of systems of cartesian modules $\FF_n$ on $(X_{n,lis-et}, \Ac)$ together with isomorphisms $\phi_\alpha\co \alpha_*\FF_m \to \FF_n$ for every injective order-preserving $\alpha\co [m] \to [n]$ such that 
 \[\phi_\alpha\circ (\alpha_*\phi_\beta) = \phi_{\alpha\circ\beta}\co  \alpha_*\beta_*\FF_k\cong (\alpha\circ\beta)_*\FF_k \to \FF_n\]
 for $\alpha\co[m]\to [n]$ and $\beta\co [k]\to [m]$ injective. Then the functor
\[\QCoh(\XX, \Ac) \to \Mod_{\cart}(X^+_\bullet)\]
given by pullback is an equivalence of categories. 
\end{thm}

We now want to discuss completions of stacks in the style of \cite{Con}.

\begin{prop}[\cite{LMB00}, 14.2.]There is a natural bijection between ideal sheaves on $\XX$ and closed substacks, sending an ideal $\II$ to a substack $V(\II)$.\end{prop}

\begin{defi}\label{DefCompl}Let $\XX$ be a locally noetherian Artin stack and $\II$ an ideal sheaf. The (formal) \textit{completion} of $\XX$ along $\II$ is defined to be the ringed site $\hat{\XX}$ with the same underlying site as $\XX$ and $\OO_{\hat{\XX}}$ given as $\varprojlim \OO_\XX/\II^n$. By the correspondence between ideal sheaves and closed substacks, we can also speak of the (formal) completion of $\XX$ along a closed algebraic substack. \end{defi}

From now on, let $\XX$ always be noetherian. As above, for a chosen ideal sheaf $\II$ we denote by $\QCoh(\hat{\XX})$ the category of cartesian $\OO_{\hat{\XX}}$-modules and by $\Coh(\hat{\XX})$ the full subcategory of modules locally of finite presentation. Note that $\hat{\XX}$ is flat since the completion $\hat{A} \to \hat{B}$ with respect to an ideal of $A$ is flat if $A\to B$ is flat for $A$ and $B$ noetherian. 

We like now to recall a variant of a result by Ferrand--Raynaud (\cite{F-R70}): 
\begin{prop}\label{F-R}Let $A$ be a noetherian ring,  $I =(t) \subset A$ a principal ideal and $\hat{A}$ the completion of $A$ with respect to $I$. Then we have cartesian squares of categories:
\[\xymatrix{\QCoh(\Spec A) \ar[r]\ar[d] & \QCoh(U) \ar[d] \\
\QCoh(\Spec \hat{A}) \ar[r] & \QCoh(\hat{U}) } \]
and
\[\xymatrix{\Coh(\Spec A) \ar[r]\ar[d] & \Coh(U) \ar[d] \\
\Coh(\Spec \hat{A}) \ar[r] & \Coh(\hat{U}) } \]
Here $U$ and $\hat{U}$ denote the complements of $V(I)$ in $\Spec A$ and $\Spec \hat{A}$, respectively. The equivalence of the fiber product to the upper left corner can be chosen functorial in $(A,I)$. \end{prop}
\begin{proof}
 The map $\Spec \hat{A} \sqcup U \to \Spec A$ is a faithfully flat map of affine schemes. Thus, by \cite[Proposition 4.2]{F-R70}, we have a cartesian square
\[\xymatrix{\QCoh(\Spec A) \ar[r]\ar[d] & \QCoh(U) \ar[d] \\
 \QCoh(U)\times \QCoh(\Spec \hat{A}) \ar[r] & \QCoh(U) \times \QCoh(\hat{U}) }
\]
Here, the lower horizontal map is just the identity on the first factor and the map induced by the inclusion on the second. This gives the first square. The functoriality follows from the construction of the $A$-module as a pullback. 

To deduce the second square from the first, we only need to prove that the functor 
\[\QCoh(\Spec A) \to \QCoh(\Spec \hat{A}) \times \QCoh(U)\simeq \QCoh(\Spec \hat{A}\sqcup \QCoh(U))\]
 detects coherence. This holds since $\Spec\hat{A} \sqcup U \to \Spec A$ is fpqc (since $A$ is noetherian) and pullback along fpqc-morphisms detects coherence. 
\end{proof}

We will generalize this result to algebraic stacks. 

\begin{defi}A substack $\XX_0$ of $\XX$ is called \textit{locally principal closed} if there is a fpqc-covering $\{U_i \xrightarrow{f_i} \XX\}$ such that $\XX_0 \times_\XX U_i$ is the vanishing locus of a $t_i\in H^0(U_i; \OO_{U_i})$. This is a generalization of the notion of an effective Cartier divisor. 
\end{defi}

For a cosimplicial ring $A^\bullet$, we define a category $\Mod_{\cart}(A_+^\bullet)$ as follows: An object is a family $(M^n)_{n\geq 0}$ of $A^n$-modules together with an isomorphism $\phi_\alpha\co A^n\tensor_{A^m}M^m\to M^n$ for every injective order-preserving $\alpha\co [m]\to [n]$ such that the $\phi_\alpha$ are compatible with compositions as in Theorem \ref{OlssonDes}. Morphism are compatible families of morphisms $M^n \to N^n$ of $A^n$-modules. Define the full subcategory $\fMod_{\cart}(A_+^\bullet)$ on objects $(M^n)_{n\in \N}$ where every $M^n$ is a finitely presented $A^n$-module.

Before we use this definition, we have to define a preliminary lemma. 

\begin{lemma}\label{AffineEquivalence}Let $(t)\subset A$ be an ideal of a noetherian ring. Then $\Gamma\co \QCoh(\widehat{\Spec A}) \to \Mod(\hat{A})$ and $\Gamma\co \QCoh(\widehat{\Spec A}[\frac1t]) \to \Mod(\hat{A}[\frac1t])$ are equivalences. Here, $A$ and $\Spec A$ are (formally) completed with respect to $(t)$ -- the latter in the sense of Definition \ref{DefCompl}; furthermore, $\widehat{\Spec A}[\frac1t]$ is the ringed site with sheaf of rings $\OO_{\widehat{\Spec A}}[\frac1t]$ on $(\Spec A)_{lis-et}$. 

These equivalences restrict to equivalences \[\Gamma\co \Coh(\widehat{\Spec A}) \to \fMod(\hat{A})\] and \[\Gamma\co \Coh(\widehat{\Spec A}[\frac1t]) \to \fMod(\hat{A}[\frac1t])\] to the finitely generated modules.\end{lemma}
\begin{proof}First, we will show that $\Gamma\co \QCoh(\widehat{\Spec A}) \to \Mod(\hat{A})$ is exact. It is enough to show that a map $M\to N$ of $\hat{A}$-modules is surjective if $M\tensor_{\hat{A}} \hat{B} \to N\tensor_{\hat{A}} \hat{B}$ is surjective for some smooth surjective map $\Spec B \to\Spec A$ (where $B$ is also completed with respect to $(t)$). This is true since $\Spec \hat{B} \to \Spec \hat{A}$ is fpqc. 

The functor $\Gamma$ also commutes with arbitrary direct sums. Indeed, we can use \cite[II.1.5.1]{Tam94} to see that the lisse-\'etale site of $\Spec A$ is noetherian and for noetherian sites, all sheaf cohomology commutes with arbitrary direct sums by \cite[I.3.11.2]{Tam94}. 

There is a functor $\Lambda\co \Mod(\hat{A}) \to \QCoh(\widehat{\Spec A})$ associating to every $\hat{A}$-module $M$ the sheafification of the presheaf associating to every smooth $\Spec A$-scheme $X$ the module \[\OO_{\widehat{\Spec A}}(X) \tensor_{\hat{A}} M.\] Clearly $\Gamma(\Lambda(\bigoplus \hat{A})) \cong \bigoplus \hat{A}$ for arbitrary direct sums of $\hat{A}$ (which are flat $\hat{A}$-modules) and thus also $\Gamma(\Lambda(M)) \cong M$ for every $\hat{A}$-module $M$ since $\Gamma$ and $\Lambda$ are both right exact. Furthermore, $\Lambda\Gamma(\FF) \cong \FF$ since $\FF$ is cartesian. Both $\Lambda$ and $\Gamma$ are compatible with the action of $A$. Thus, $\Gamma$ restricts to an equivalence of the full subcategories where $t$ acts invertible, yielding the second equivalence. 
\end{proof}

Since we are only interested in applications to coherent modules, we will state the remaining results only for these.

\begin{thm}Let $\XX$ be a noetherian separated Artin stack and $\XX_0$ be a locally principal closed substack of $\XX$. Let $U$ be the complement of $\XX_0$ and $\hat{\XX}$ the completion of $\XX$ along $\XX_0$. Furthermore, let $\hat{U}$ be the ringed site with same underlying site as $\XX$ and $\OO_{\hat{U}} = \OO_{\hat{\XX}}[\frac1t]$ locally for $\XX_0$ the vanishing locus of $t$. 

Let $\Spec A \to \XX$ be a smooth cover such that $\Spec A\times_\XX \XX_0$ is cut out by a single element $t\in A$. Define a cosimplicial ring $A^\bullet$ by $\Spec A^n = \Spec A^{\times_\XX n+1}$. Then pullback defines equivalences
 \begin{align*} \Coh(\XX) &\to \fMod_{\cart}(A^\bullet_+) \\
  \Coh(\hat{\XX}) &\to \fMod_{\cart}(\widehat{A^\bullet}_+)\\
 \Coh(\hat{U}) &\to \fMod_{\cart}(\widehat{A^\bullet}[\frac1t]_+)
 \end{align*}
\end{thm}
\begin{proof}
 The first equivalence is clear by Theorem \ref{OlssonDes} since $\Spec A \to \XX$ is a fpqc map and detects therefore coherence. The second and the third equivalence follow in the same way from Theorem \ref{OlssonDes}, the fact that completion and localization preserve flatness and the previous lemma. 
\end{proof}

\begin{thm}\label{glueing}Let $\XX$ be a noetherian separated Artin stack and $\XX_0$ be a locally principal closed substack of $\XX$. Let $U$ be the complement of $\XX_0$ and $\hat{\XX}$ the completion of $\XX$ along $\XX_0$. Furthermore, let $\hat{U}$ be the ringed site with same underlying site as $\XX$ and $\OO_{\hat{U}} = \OO_{\hat{\XX}}[\frac1t]$ locally for $\XX_0$ the vanishing locus of $t$. Then we have cartesian squares
\[\xymatrix{ \Coh(\XX) \ar[r]\ar[d] & \Coh(U) \ar[d] \\
\Coh(\hat{\XX}) \ar[r] & \Coh(\hat{U}). } \]
and
\[\xymatrix{ \Vect(\XX) \ar[r]\ar[d] & \Vect(U) \ar[d] \\
\Vect(\hat{\XX}) \ar[r] & \Vect(\hat{U}). } \]
\end{thm}
\begin{proof}We begin by showing the first square to be cartesian. This is true if $\XX$ is an affine scheme and $\XX_0$ is cut out by a single element by Proposition \ref{F-R} and Lemma \ref{AffineEquivalence}. The general case follows from the last theorem by choosing a smooth cover $\Spec A \to \XX$ where $\XX_0$ is cut out by a single element $t$. 

A coherent sheaf on $\XX$, $\hat{\XX}$ or $\hat{U}$ is a vector bundle iff its restriction to $(\Spec A)_{\text{lis-et}}$ is. Since $\Spec \hat{A} \sqcup \Spec A[\frac1t] \to \Spec A$ is fpqc and being locally free is detected by fpqc maps, the second square follows. \end{proof}

\bibliographystyle{alpha}
\bibliography{Chromatic}
\end{document}